\documentclass[hidelinks, 12pt]{amsart}
\usepackage{preamble}

\usepackage{geometry}
\title{Murphy's Law for Galois Deformation Rings}

\author{Andreea Iorga}

\address{The University of Chicago, 5734 S University Ave, Chicago, IL 60637, USA}
\email{aiorga@uchicago.edu}

\usepackage{microtype}

\begin{document}

\begin{abstract}
	In this paper, we prove, under a technical assumption, that any semi-direct product of a $p$-group $G$ with a group $\Phi$ of order prime to $p$ can appear as the Galois group of a tower of extensions $H/K/F$ with the property that $H$ is the maximal pro-$p$ extension of $K$ that is unramified everywhere, and $\Gal(H/K) = G$. A consequence of this result is that any local ring admitting a surjection to $\Z_5$ or $\Z_7$ with finite kernel can occur as a universal everywhere unramified deformation ring. 
\end{abstract}

\maketitle

\section{Introduction}

Let $p$ be a prime. Let $\Phi$ be a finite group of order prime to $p$, and let $G$ be a finite $p$-group with an action of $\Phi$. Throughout this paper, $\Phi$ will be fixed, and $G$ will represent any $p$-group with an action of $\Phi$. Let $\Gamma = G \rtimes \Phi$ be the semi-direct product of $G$ and $\Phi$. 
For any number field $F$, let $L_p(F)$ denote the maximal unramified $p$-extension of $F$. If $\Phi$ is the trivial group, Ozaki's Theorem (Theorem 1 in \cite{MR2772089}) states that any $p$-group $\Gamma = G$ can be written as the Galois group of $L_p(F)/F$, for some totally complex number field $F$. A recent paper by Hajir, Maire and Ramakrishna (\cite{hajir2022ozakis}) provides two extensions to Ozaki's result: the base field can have arbitrary signature, as long as its class number is prime to $p$, and the degree of the new field over $\Q$ can be controlled. In this paper, we prove a different generalization in the case of regular primes: 

\begin{imptheorem} \label{thm1}
	Let $p$ be a prime. Let $\Phi$ be a group of order prime to $p$. Assume there exists an extension of number fields $L/E$ such that 
	\begin{itemize}
		\item $L/E$ is Galois with Galois group $\Phi$,
		\item $E$ contains $\mu_p$, and is totally imaginary if $p = 2$,
		\item $L$ has class number prime to $p$,
		\item $L/E$ satisfies property \textbf{P} below. 
	\end{itemize}
	For any $p$-group $G$ with an action of $\Phi$, there exist extensions of number fields $H/K/F$ such that
	\begin{enumerate}
		\item $H/K$ is the maximal pro-$p$ extension of $K$ that is unramified everywhere,
		\item $\Gal(H/K) = G$,
		\item $\Gal(H/F) = \Gamma$, where $\Gamma = G \rtimes \Phi$, 
		\item $H/F$ satisfies property \textbf{P} below. 
	\end{enumerate}
\end{imptheorem}

\begin{definition}
	We say that an extension of number fields $L/K$ has property \textbf{P} if for all primes $\mathfrak{p}$ of K, and $\mathfrak{P} \mid \mathfrak{p}$, either $L_{\mathfrak{P}}/K_{\mathfrak{p}}$ is unramified or $L_{\mathfrak{P}}/K_{\mathfrak{p}}$ is tamely ramified with ramification index $e$ and $e \mid (q-1)$, where $q$ is the cardinality of the residue field of $K_{\mathfrak{p}}$. 
\end{definition}

When $\Phi$ is trivial, we can recover Ozaki's result in the case when $p$ is a prime such that $\Q(\zeta_p)$ has a finite extension with class number prime to $p$ (note that this includes regular primes); a similar hypothesis is present in the first version (arXiv:0705.2293) of Ozaki's paper \cite{MR2772089}. The proof of this theorem is presented in Section \ref{strategythm1}, and it is inspired by Ozaki's theorem and techniques. Throughout this paper, we will present the similarities and differences between our methods and Ozaki's methods.

A motivating example and a consequence of Theorem \ref{thm1} is Theorem \ref{thm2} below. Consider a continuous absolutely irreducible residual Galois representation $\overline{\rho} \colon G_F \to \GL_2(\F_p)$. One can associate to $\overline{\rho}$ a number of deformation rings. These pro-represent functors of deformations from the category $\mathcal{C}$ of local Artinian rings $(A, \mathfrak{m})$ with $A/\mathfrak{m} = \F_p$. 
\begin{definition}
	Consider a deformation $\rho \colon G_F \to \GL_2(A)$ of $\overline{\rho}$ for a finite $(A, \mathfrak{m})$. It factors through some finite group, and the fixed field of the kernel is a finite extension; call it $F(\rho)$. We say that $\rho$ is unramified if the extension $F(\rho)/F(\overline{\rho})$ is unramified everywhere. 
\end{definition}{}
The functor on $\mathcal{C}$ which sends $A$ to the unramified deformations $D(A)$ is pro-representable by a universal deformation ring. We are interested in the following question: 

\begin{question}
	What possible rings $R$ can occur as universal everywhere unramified deformation rings of such $\overline{\rho}$?
\end{question}

Assume that the image of $\overline{\rho}$ has order prime to $p$, so its projective image is $\Phi = A_4, \text{ } S_4, \text{ } A_5$ or a dihedral group (Proposition 16 in \cite{MR0387283}).
The Unramified Fontaine-Mazur Conjecture (Conjecture 5a in \cite{MR1363495}) 
predicts that all $\overline{\Q}_p$-points will have finite image. Moreover, the tangent space to any $\overline{\Q}_p$-point with finite image will
be trivial by class field theory (proof of Proposition 10 in \cite{MR3294389}), and thus conjecturally such a ring
has a unique map to $\overline{\Q}_p$. The expectation is then that $R$ is a ring admitting a map $R \to \Z_p$ with finite (as a set) kernel $I$. In this paper, we prove the following: 

\begin{imptheorem} \label{thm2}
	Let $R$ be any local ring admitting a surjection to $\Z_5$ or to $\Z_7$ with finite kernel. Then $R$ is a universal everywhere unramified deformation ring. 
\end{imptheorem}

This result can be seen as an example of Murphy's Law for moduli spaces, an idea introduced by Ravi Vakil in \cite{MR2227692}: all possible singularities occur inside deformation spaces. When considering unramified deformation rings, the analogue of this is to say that all finite artinian local rings appear as unramified Galois deformation rings. 

We now outline the structure of the proof, and the differences and similarities to Ozaki's methods. The proof of Theorem \ref{thm1} is done by induction, as follows. Since $\Phi$ acts on the $p$-group $G$, it must preserve the centre $Z(G)$ and the $p$-torsion of the centre of this group. It follows that each such $p$-group fits into an exact sequence of $p$-groups
\[ 1 \to V \to G^\prime \to G \to 1,\]
where $V$ is a central subgroup of exponent $p$ on which $\Phi$ acts by an irreducible representation. Therefore, there exists a sequence of $p$-groups
\[
G = G_n \to G_{n-1} \to \dots \to G_0 = 1,
\]
where each map is surjective and $\ker (G_i \to G_{i-1}) = V$ at each step. It follows that there exists a sequence of surjections 
\[
\Gamma = \Gamma_n \to \Gamma_{n-1} \to \dots \to \Gamma_0 = \Phi,
\]
such that the kernel at each step is isomorphic to $V$, where $\Gamma_i = G_i \rtimes \Phi$. The base case of the inductive process is the assumption of Theorem \ref{thm1}. The inductive step follows from Proposition \ref{prop2}, whose proof is presented in Section \ref{proofofprop2}. Just as in Ozaki's case, the extensions are constructed using Kummer Theory. The main difference between our proof and Ozaki's is that we are not working with $p$-groups, but with $p$-groups with a $\Phi$-action. Thus, we need to construct a big number of primes satisfying a series of congruence conditions. Constructing enough primes relies on the fact that the base field has a large enough degree over $\Q$. To ensure this, we perform a series of base changes using Proposition \ref{prop1}, whose proof is presented in Section \ref{proofofProp1}. The proof relies on the theory of modular representations of $\F_p[\Gamma]$. Finally, the proof of Theorem \ref{thm2} can be found in the last section. In fact, Section \ref{proofofthm2} presents a proof that works in general for any prime $p \geq 5$ with the assumption that there exists a $\Phi$-extension of $\Q(\zeta_p)$ with class number prime to $p$ that satisfies property \textbf{P}. 

\section{Strategy for proving Theorem \ref{thm1}} \label{strategythm1}

Theorem \ref{thm1} can be derived from the following two propositions, which will be proved in Sections \ref{proofofProp1} and \ref{proofofprop2}. 

\begin{proposition} \label{prop1}
	With the above notation, let $K/F$ be a Galois extension of number fields with Galois group $\Phi$ satisfying:
	\begin{itemize}
		\item The extension $L_p(K)/F$ is Galois and has Galois group isomorphic to $\Gamma = G \rtimes \Phi$,
		\item The field $F$ contains the group $\mu_p$, and is totally imaginary if $p=2$, 
		\item Every prime of $F$ lying over $p$ splits completely in $L_p(K)$,
		\item The extension $F/\Q$ satisfies $[F\colon \Q] \geq 2\left(2d(G) + r(G) + d(\Phi)\right)$, where $d(\tilde{G})$ and $r(\tilde{G})$, respectively, are the minimal numer of generators and relations of a group $\tilde{G}$. 
	\end{itemize}
	Then there exists a cyclic extension $F^\prime/F$ of degree $p$ such that if $K^\prime = F^\prime.K$, then: $F^\prime \cap L_p(K) = F$, $L_p(K^\prime) = F^\prime.L_p(K)$ and $\Gal(L_p(K^\prime)/F^\prime) \cong \Gamma$. Moreover, if the initial extension $K/F$ satisfies property \textbf{P}, then the new extension $K^\prime/F^\prime$ also satisfies property \textbf{P}. 
\end{proposition} 

\begin{proposition}\label{prop2}
	Let $K/F$ be a Galois extension of number fields satisfying the four conditions of Proposition \ref{prop1} and property \textbf{P}. Assume $\Phi$ acts irreducibly on a $p$-group $V$. Then for any exact sequence of groups
	\[
	1 \to V \to \Gamma^\prime \to \Gamma \to 1,
	\]
	there exists a finite extension of fields $K^\prime/F^\prime$ such that
	\begin{enumerate}
		\item $F \subset F^\prime$ and $K \subset K^\prime$,
		\item The extension $K^\prime/F^\prime$ is Galois and has Galois group isomorphic to $\Phi$,
		\item The extension $L_p(K^\prime)/F^\prime$ is Galois and has Galois group isomorphic to $\Gamma^\prime$,
		\item Every prime of $F^\prime$ lying over $p$ splits completely in $L_p(K^\prime)$,
		\item The extension $L_p(K^\prime)/F^\prime$ satisfies property \textbf{P}.
	\end{enumerate}
\end{proposition}

The proof of Theorem \ref{thm1} is inspired by Ozaki's results, and follows from Propositions \ref{prop1} and \ref{prop2} by induction. Recall that for a $p$-group $G$ with a $\Phi$-action, we have an exact sequence of $p$-groups
\[
G = G_n \to G_{n-1} \to \dots \to G_0 = 1,
\]
where each map is surjective and the kernel at each step is isomorphic to $V$. If $\Gamma_i = G_i \rtimes \Phi$, then we have a sequence of surjections
\[\Gamma = \Gamma_n \to \Gamma_{n-1} \to \dots \to \Gamma_0 = \Phi,\]
with $\ker(\Gamma_i \to \Gamma_{i-1}) \cong V$, for $1 \leq i \leq n$. The assumption of Theorem \ref{thm1} is the base case of our inductive proof. At step $i$, we can assume that we have a Galois extension $K_i/F_i$ satisfying the conditions of Proposition \ref{prop1} for $G_i$ and $\Gamma_i$. Using Proposition \ref{prop1} repeatedly, we can construct a finite extension $F_i^\prime$ of $F_i$ such that if $K_i^\prime = F_i^\prime.K_i$, then $F_i^\prime \cap L_p(K_i) = F_i$, $L_p(K_i^\prime) = F_i^\prime.L_p(K_i)$, and $\Gal(L_p(K^\prime_i)/F_i^\prime) \cong \Gamma_i$. Moreover, repeatedly constructing extensions using Proposition \ref{prop1}, we increase the degree $[F_i^\prime \colon \Q]$, while keeping $2(2d(G_{i+1}) + r(G_{i+1}) + d(\Phi))$ unchanged. Thus, we can also assume that $[F_i^\prime \colon \Q] \geq 2(2d(G_{i+1}) + r(G_{i+1}) + d(\Phi))$. Since this extension $K_i^\prime/F_i^\prime$ satisfies the conditions of Proposition \ref{prop2}, there exists a finite extension $K_{i+1}/F_{i+1}$ such that $F_i^\prime \subset F_{i+1}$, $K_i^\prime \subset K_{i+1}$, $\Gal(K_{i+1}/F_{i+1}) \cong \Phi$, $\Gal(L_p(K_{i+1})/F_{i+1}) \cong \Gamma_{i+1}$, every prime of $F_{i+1}$ lying over $p$ splits completely in $L_p(K_{i+1})$ and $L_p(K_{i+1})/F_{i+1}$ satisfies property \textbf{P}. Therefore, we have obtained fields $F = F_n$, $K = K_n$ and $H = L_p(K_n)$ with the desired properties. 

\section{Tools for the proof}

In this section, we present some facts that will be useful later in the paper. Most of these results can either be found in \cite{MR2772089} or are generalizations of results in \cite{MR2772089}. We will follow Ozaki's notation. 

Suppose $F$ is a number field. Let $U_{\mathfrak{p}}(F)$ be the pro-$p$-part of the local unit group of the complete field $F_{\mathfrak{p}}$ of $F$ at $\mathfrak{p}$ and $U(F) = \oplus_{\mathfrak{p} \mid p} U_{\mathfrak{p}}(F)$. We embed the unit group of the localisation $\mathcal{O}_{F, \mathfrak{p}}$ of the maximal order $\mathcal{O}_F$ of $F$ at $p$ diagonally into $U(F)$ as usual, and let $U_{\mathfrak{p}}^\prime(F)$ be the submodule of $U_{\mathfrak{p}}(F)$ consisting of all the elements $u$ such that $F_{\mathfrak{p}}(\sqrt[p]{u})/F_{\mathfrak{p}}$ is unramified; let $U^\prime(F) = \oplus_{\mathfrak{p} \mid p} U_{\mathfrak{p}}^\prime(F)$. Since we have that $U(F)^p \subset U^\prime(F) \subset U(F)$, we can define $R(F) = U(F)/U(F)^p$ and $R^\prime(F) = U(F)/U^\prime(F)$. 

\begin{lemma}\label{lemmafree}
	If $K/F$ is a Galois extension with Galois group $\Phi$ as above, then 
	\begin{enumerate}
		\item $R(K) \cong \F_p[\Phi]^{[F\colon \Q] + s}$ and $R^\prime(K) \cong \F_p[\Phi]^{[F \colon \Q]}$,
		\item $R(L_p(K)) \cong \F_p[\Gamma]^{[F\colon\Q] + s}$ and $R^{\prime}(L_p(K)) \cong \F_p[\Gamma]^{[F\colon \Q]}$,
	\end{enumerate}
	where $s$ is the number of primes of $F$ lying over $p$. 
\end{lemma}
\begin{proof}
	By definition, $U(F) = \oplus_{\mathfrak{p} \mid p} U_\mathfrak{p} (F) = \oplus_{\mathfrak{p} \mid p} \mathcal{O}_{F_\mathfrak{p}}^\times \otimes_{\Z_p} \Z_p$. Tensoring with $\F_p$, we obtain that
	$\left(\mathcal{O}_{F_\mathfrak{p}}^\times \otimes_{\Z_p} \Z_p\right)\big/\left(\mathcal{O}_{F_\mathfrak{p}}^\times \otimes_{\Z_p} \Z_p\right)^p \cong \F_p^{d_{\mathfrak{p}}+1}$, where $d_{\mathfrak{p}} = [F_\mathfrak{p}\colon \Q_p]$. Note that $\sum_{\mathfrak{p}\mid p} d_{\mathfrak{p}} = [F\colon \Q]$. It follows that $R(F) \cong \F_p^{[F\colon\Q]+s}$. Similarly, $R^\prime(F) \cong \F_p^{[F\colon\Q]}$.

	Because every prime of $F$ lying over $p$ splits completely in $K/F$, we have a natural isomorphism of $\Phi$-modules $U(K) \cong \Z_p[\Phi] \otimes_{\Z_p} U(F)$ and $U^\prime(K) \cong \Z_p[\Phi] \otimes_{\Z_p} U^\prime(F)$, which shows that $R(K) \cong \F_p[\Phi]^{[F\colon\Q] + s}$ and $R^\prime(K) \cong \F_p[\Phi]^{[F\colon \Q]}$. 

	Similarly, since every prime of $F$ lying over $p$ splits completely in $L_p(K)/F$, we obtain that $R(L_p(K)) \cong \F_p[\Gamma]^{[F\colon\Q] + s}$ and $R^\prime(L_p(K)) \cong \F_p[\Gamma]^{[F\colon \Q]}$. 
\end{proof}

\begin{lemma} \label{lemma1}
	Let $p$ be any prime number, $F$ a number field with $L_p(F) = F$ and $S$ a finite set of primes of $F$. We denote by $F_S/F$ the maximal elementary abelian $p$-extension of $F$ unramified outside $S$. For any prime $v$ of $F$, denote by $D_v$ the decomposition subgroup of $\Gal(F_S/F)$ at the prime $v$. We assume that the map 
	\begin{align*}
		\bigoplus_{v} H_2(D_v, \Z) \to H_2 (\Gal(F_S/F), \Z)
	\end{align*}
	induced by the natural inclusion $D_v \subset \Gal(F_S/F)$ is surjective. Then $L_p(F_S) = F_S$. 
\end{lemma}
\begin{proof}
	See Lemma 7 in \cite{MR2772089}. 
\end{proof}
\begin{corollary}\label{cor1}
	Let $S$ and $F_S$ be as in Lemma \ref{lemma1}. If $F_S/F$ is a cyclic extension, then $L_p(F_S) = F_S$. 
\end{corollary}

The following is a variant of Lemma 9 of \cite{MR2772089}, which only applies to $M = L_p(K)$. Our modification works for both $M = L_p(K)$ and $M = K$.
\begin{lemma}\label{lemma3}
	Let $L_p(K)/F$ be an extension as above. Let $M = K$ or $L_p(K)$, and let $\tilde{M}/M$ be any finite abelian extension linearly disjoint from the maximal abelian extension of $M$ unramified outside $p$. Then for any $u \in R(M)$ and any $\tau \in \Gal(\tilde{M}/M)$, there exist infinitely many prime ideals $\Lambda\mathcal{O}_M$ of $\mathcal{O}_M$ such that $\Lambda\mathcal{O}_M$ is prime to $p$, $(\Lambda \mod{U(M)^p)} = u$ in $R(M)$, and $(\Lambda \mathcal{O}_M, \tilde{M}/M) = \tau$.  
\end{lemma}
\begin{proof}
	Let $L$ be the maximal elementary abelian $p$-extension of $M$ which is unramified outside $p$, and let $H$ be the maximal elementary abelian $p$-extension of $M$ unramified everywhere. Then we have the following exact sequence 
	\[
	\mathcal{O}_M^\times \otimes \F_p \to R(M) \xrightarrow[{}]{\rho} \Gal(L/M) \xrightarrow[{}]{f} \Gal(H/M) \to 1,
	\]
	where the map $\rho \colon R(M) \to \Gal(L/M)$ is the map induced by class field theory, and the third map is the natural surjection $f \colon \Gal(L/M) \to \Gal(H/M)$. Let $\sigma = \rho(u) \in \Gal(L/M)$. Let $N$ be the maximal unramified abelian extension of $M$. Note that $L$ and $N$ are linearly disjoint over $H$, and let $\tilde{L}$ be their compositum. Observe that $\tilde{L}$ and $\tilde{M}$ are linearly disjoint over $M$. Let $\tilde{\sigma} \in \Gal(\tilde{L}/M)$ be an element with the properties that $\res(\tilde{\sigma})\mid_L = \sigma^{-1}$ and $\res(\tilde{\sigma})\mid_N = 1$. Such an element exists because $L$ and $N$ are linearly disjoint over $H$, and the restrictions $\sigma^{-1} \in \Gal(L/M)$ and $1 \in \Gal(H/M)$ agree on $H/M$, since $\res(\sigma^{-1})\mid_H = \res(1)\mid_H$ if and only if $\sigma^{-1} \in \ker(f) = \im(\rho)$, which is true by construction. 

	By the Chebotarev density theorem, there are infinitely many degree one primes $\alpha$ of $\mathcal{O}_{M}$ not lying over $p$ such that $(\alpha, \tilde{L}/M) = \tilde{\sigma}$ and $(\alpha, \tilde{M}/M) = \tau$. The first condition implies that $(\alpha, L/M) = \sigma^{-1}$ and $(\alpha, N/M) = 1$. The second property implies that $\alpha$ is a principal ideal in $\mathcal{O}_M$, so there exists $\Lambda_0 \in M$ such that $\alpha = \Lambda_0 \mathcal{O}_M$. Combining this with the first condition, we obtain that $\Lambda_0 = \Lambda \varepsilon$, for some $\varepsilon \in \mathcal{O}_M^\times$ and some $\Lambda \in M$. The element $\Lambda$ has the properties $(\Lambda \mod{U(M)^p}) = u$ in $R(M)$ and $(\Lambda\mathcal{O}_M, \tilde{M}/M) = \tau$, which is what we wanted. 
\end{proof}

\section{Proof of Proposition \ref{prop1}} \label{proofofProp1}

In this section, we provide a proof for Proposition \ref{prop1}, which is our  version of Proposition 1 in the first version of \cite{MR2772089}. Our proof follows the idea of Ozaki's proof, modified to work in our situation. More explicitly, in his proof, Ozaki uses the theory of $\F_p$-representations of $p$-groups $G$, while we have to use the theory of $\F_p$-representations of groups of the form $G \rtimes \Phi$, where $G$ is a $p$-group and $\Phi$ is a group of order prime-to-$p$. Throughout this section, assume that the conditions of Proposition \ref{prop1} hold.  

We would like to find an element $\Lambda$ of $L_p(K)$ such that
\begin{enumerate}
	\item The ideal $\Lambda \mathcal{O}_{L_p(K)}$ is a prime ideal of degree $1$, not lying over $p$. \label{cond1}
	\item If $S$ denotes the set of primes of $L_p(K)$ dividing $\eta = N_{L_p(K)/F}(\Lambda)$, then $L_p(K)(\sqrt[p]{\eta})$ is the maximal elementary abelian $p$-extension of $L_p(K)$ unramified outside $S$. \label{cond2}
\end{enumerate}

\begin{lemma} \label{lemmma9}
	Assume such an element $\Lambda$ exists. Let $F^\prime = F(\sqrt[p]{\eta})$, with $\eta$ as above. Then $F^\prime/F$ is an extension that satisfies Proposition \ref{prop1}. 
\end{lemma}
\begin{proof}
	Since $\eta \mathcal{O}_F$ is a prime ideal of $\mathcal{O}_F$, it follows that $\sqrt[p]{\eta} \not\in F$, so $F^\prime$ is a degree $p$ extension of $F$. Let $K^\prime = K.F^\prime$. The fields $K$ and $F^\prime$ are linearly disjoint over $F$, so $K^\prime/K$ is a degree $p$ extension. We want to prove that $F^\prime \cap L_p(K) = F$, $L_p(K^\prime) = F^\prime . L_p(K)$, and that the extension $L_p(K^\prime)/F^\prime$ is Galois with Galois group isomorphic to $\Gamma$. 

	Consider $F^\prime \cap L_p(K)$. This is equal to $F^\prime$ if $F^\prime \subset L_p(K)$; otherwise, it is equal to $F$. Assume $F^\prime \subset L_p(K)$. By construction, this implies that $K^\prime \subset L_p(K)$. But $L_p(K)$ is the maximal unramified $p$-extension of $K$, and $K^\prime$ is a ramified $p$-extension of $K$, so they must be linearly disjoint over $K$, and $K^\prime \not\subset L_p(K)$. It follows that our assumption was false, and so $F^\prime \cap L_p(K) = F$, which proves the first part. 

	Using the previous part, we observe that $F^\prime.L_p(K) = L_p(K)(\sqrt[p]{\eta})$ and $K^\prime.L_p(K) = L_p(K)(\sqrt[p]{\eta})$. By construction, $L_p(K)(\sqrt[p]{\eta})$ is the maximal elementary abelian $p$-extension of $L_p(K)$ unramified outside $S$, so Corollary \ref{cor1} tells us that $L_p(L_p(K)(\sqrt[p]{\eta})) = L_p(K)(\sqrt[p]{\eta})$. On one hand, since $K^\prime \subset K^\prime.L_p(K)$, we must have $L_p(K^\prime) \subset L_p(K^\prime.L_p(K)) = K^\prime.L_p(K)$. On the other hand, $K^\prime.L_p(K)$ is an unramified $p$-extension of $K^\prime$, so $K^\prime.L_p(K) \subset L_p(K^\prime)$. Combining these remarks, we obtain that $L_p(K^\prime) = K^\prime.L_p(K) = F^\prime.L_p(K)$, proving the second part. 

	Finally, consider the following diagram 
	\[
    \begin{tikzcd}[arrows=dash]
	& L_p(K^\prime) = F^\prime.L_p(K) \dlar \drar\\
L_p(K) \drar& & F^\prime \dlar\\
& F
\end{tikzcd}
\]

Since the extension $L_p(K)/F$ is Galois and has Galois group isomorphic to $\Gamma$, it follows that $L_p(K^\prime)/F^\prime$ is Galois and $\Gal(L_p(K^\prime)/F^\prime) = \Gal(F^\prime.L_p(K)/F^\prime) \cong \Gal(L_p(K)/F) \cong \Gamma$, which is what we wanted. 

To conclude the proof, assume that the initial extension $K/F$ has property \textbf{P}. Let $\mathfrak{p}^\prime$ be any prime of $F^\prime$ and $\mathfrak{p}$ be a prime of $F$ below $\mathfrak{p}^\prime$. Let $e$ and $e^\prime$ be the ramification indices of $\mathfrak{p}$ in $K/F$ and of $\mathfrak{p}^\prime$ in $K^\prime/F^\prime$, respectively. Let $q$ and $q^\prime$ be the number of elements of the residue fields of $F_{\mathfrak{p}}$ and $F^\prime_{\mathfrak{p}^\prime}$, respectively. Since the order of $\Phi$ is prime to $p$, we must have that $e=e^\prime$, and $q^\prime = q$ or $q^\prime = q^p$. Since $K/F$ has property \textbf{P}, then either $e = 1$ or $K_{\mathfrak{p}}/F_{\mathfrak{p}}$ is tamely ramified with $e \mid (q-1)$. It follows that either $e^\prime = e = 1$ or $K^\prime_{\mathfrak{p}^\prime}/F^\prime_{\mathfrak{p}^\prime}$ is tamely ramified with $e^\prime \mid (q-1) \mid (q^\prime - 1)$, so the new extension $K^\prime/F^\prime$ has property \textbf{P}. 
\end{proof}

Now, assume that $\Lambda$ has property (\ref{cond1}). Let $S$  be the set of primes of $L_p(K)$ dividing $\eta$. If $L_p(K)(\sqrt[p]{\alpha})/L_p(K)$ is unramified outside $S$, for some $\alpha \in L_p(K)$, then 
\[
\alpha \mod{L_p(K)^{\times p}} \equiv (\varepsilon \mod{L_p(K)^{\times p}}) + \sum_{\sigma\in \Gamma}a_\sigma(\sigma\Lambda  \mod{L_p(K)^{\times p}}),
\]
with $a_\sigma \in \F_p$, $\varepsilon \in \mathcal{O}^\times_{L_p(K)}$. Since $L_p(K)(\sqrt[p]{\alpha})/L_p(K)$ is unramified at the primes above $p$, it must be true that 
\[
(\varepsilon \mod{U^\prime(L_p(K))}) +  \sum_{\sigma \in \Gamma} a_\sigma(\sigma \Lambda \mod{U^\prime(L_p(K))})  \equiv 0.
\]
If this equation only holds for $\varepsilon \in (\mathcal{O}^\times_{L_p(K)})^p$ and $a_\sigma = a$, $\forall \sigma \in \Gamma$, for some $a \in \F_p$, then 
\begin{align*}
	\alpha \mod{U(L_p(K))^p} \equiv a\sum_{\sigma\in \Gamma} \sigma (\Lambda \mod{U(L_p(K))^p})  = a(\eta \mod{U(L_p(K))^p}).
\end{align*}
Thus, $\sqrt[p]{\alpha} \in L_p(K)(\sqrt[p]{\eta})$, so condition (\ref{cond2}) also holds. 

Let $E = E(L_p(K))$ be the image of the map $\mathcal{O}^\times_{L_p(K)} \otimes \F_p \to R^\prime(L_p(K))$. The extension $L_p(K)(\sqrt[p]{\varepsilon})/L_p(K)$ must be ramified at some prime lying over $p$, for any $\varepsilon \in \mathcal{O}^\times_{L_p(K)} \setminus (\mathcal{O}^\times_{L_p(K)})^p$, so this map is injective: $E \cong \mathcal{O}^\times_{L_p(K)} \otimes \F_p$. It follows that the map $\mathcal{O}^\times_{L_p(K)}\otimes \F_p \to R(L_p(K))$ is also injective, and by abuse of notation we denote its image by $E$. 

The following Lemma represents a key step in the proof of Proposition \ref{prop1}. It is a variation of Lemma 8 in \cite{MR2772089} and it is inspired by Lemma 2 in the first version of the same paper. The main difference between Ozaki's proof and our proof comes from the fact that $\F_p[G]$ is a projective indecomposable $\F_p[G]$-module, and this doesn't remain true if we replace $G$ by $\Gamma = G \rtimes \Phi$ (for $G$ a $p$-group and $\Phi$ a prime-to-$p$ group). To deal with this, we turn to the theory of modular representations for groups of the form $G \rtimes \Phi$ (\cite{MR3617363}) and we make use of some of the ideas that appear in Section 6 of \cite{hajir2017analytic}. 

\begin{lemma}\label{lemma2}
	Let $N$ be the kernel of the projection $R(L_p(K)) \to R^\prime(L_p(K))$. Then there exist free $\F_p[\Gamma]$-modules $M$, $N$, $Q$ of $R(L_p(K))$ such that 
	\begin{itemize}
		\item $R(L_p(K)) = M \oplus N \oplus Q$ and $R^\prime(L_p(K)) \cong M \oplus Q$.
		\item $E \subset M$.
		\item $\rank_{\F_p[\Gamma]} Q \geq \frac{1}{2}[F\colon \Q] - d(G) - r(G)$.
	\end{itemize}
\end{lemma}
\begin{proof}
	Note that $\F_p[\Gamma]$ is a Frobenius algebra, so injective $\F_p[\Gamma]$-modules are the same as projective $\F_p[\Gamma]$-modules (see Section 8.5 in \cite{MR3617363}). In particular, any free $\F_p[\Gamma]$-module is injective. 

	Since $\Phi$ has order prime to $p$, the projective indecomposable $\F_p[\Gamma]$-modules $P_S$ are in a one-to-one correspondence with the simple $\F_p[\Phi]$-modules $S$ (Proposition 8.3.2 in \cite{MR3617363}). It follows that $\F_p[\Gamma]$ can be decomposed as a sum of projective indecomposable modules \[\F_p[\Gamma] = \displaystyle \bigoplus_{S \text{ simple}} P_S^{n_s},\] where $n_s = \dim_D(S)$, $D = \End_{\F_p[\Phi]}S$. 

	From Section 8.5 in \cite{MR3617363}, we know that any $\F_p[\Gamma]$-module has a unique injective hull. Let $\tilde{M} = \oplus P_S^{\alpha_S}$ be the injective hull of the $\F_p[\Gamma]$-module $E$. Consider the following diagram
	\[
    \begin{tikzcd}
    E\arrow[hookrightarrow]{r}{f}\arrow[hookrightarrow]{d}{i} 
 	& \tilde{M} \arrow[dl, dashed, "g"]\\
  	R(L_p(K))
	\end{tikzcd}
	\]
The $\F_p[\Gamma]$-module $R(L_p(K))$ is free by Lemma \ref{lemmafree}, so it is injective. The map $i$ is the usual inclusion map from $E$ to $R(L_p(K))$. The map $f$ is the essential monomorphism $E \to \tilde{M}$. Since $R(L_p(K))$ is an injective $\F_p[\Gamma]$-module, there exists a map $g \colon \tilde{M} \to R(L_p(K))$ such that $g \circ f = i$. Moreover, since $f$ is an essential monomorphism, the map $g$ is injective. Let $M_1 = \im(g) \subset R(L_p(K))$; the module $E$ can be seen as a submodule of $M_1$. 

By definition of $N$, we have a short exact sequence of $\F_p[\Gamma]$-modules \[1 \to N \to R(L_p(K)) \to R^\prime(L_p(K)) \to 1.\] 
Since $R^\prime(L_p(K))$ is a free module (Lemma \ref{lemmafree}), this sequence splits. It follows that $N$ is a stably free $\F_p[\Gamma]$-module, which implies that $N$ is a free $\F_p[\Gamma]$-module, of rank $s$ (Example 4.7(3) in \cite{MR2235330}). 

Consider the intersection $M_1 \cap N$. Let $\Soc(M_1 \cap N)$ be the socle of $M_1 \cap N$. For details about this notion, see Section 6.3 of \cite{MR3617363}. Since $\F_p[\Gamma]$ is an Artinian ring, every nonzero module has a simple submodule. It follows that if $M_1 \cap N$ is nonzero, then $\Soc(M_1 \cap N)$ must be nonzero. Moreover, $M_1 \cap N$ is a submodule of $M_1$. Using the facts that $\Soc E \cong \Soc M_1$, $\Soc K \subset K$ and $\Soc K = K \cap \Soc M_1$, for all submodules $K$ of $M_1$, we obtain
\begin{align*}
	\Soc(M_1 \cap N) &=(M_1 \cap N) \cap \Soc M_1 \\ &= N \cap \Soc M_1 \\
	&\cong N \cap \Soc E \\ &\subset N \cap E \\&= 0. 
\end{align*}
It follows that $M_1 \cap N = 0$, so $M_1 + N$ is a direct sum in $R(L_p(K))$. Since $M_1 \oplus N$ is a projective $\F_p[\Gamma]$-module, it must also be injective, so the following exact sequence splits: 
\begin{align*}
	1 \to M_1 \oplus N \to R(L_p(K)) \to R(L_p(K))/(M_1 \oplus N) \to 1.
\end{align*}
Let $Q_1 = R(L_p(K))/(M_1 \oplus N)$. This is a projective $\F_p[\Gamma]$-module, so it can be written as $Q_1 = \oplus P_S^{\beta_S}$ with the property that $\alpha_S + \beta_S = [F\colon \Q] \cdot n_S$.

We would like to estimate $\beta_S$. Let $r = \rank_{\F_p[\Phi]}E^G = \rank_{\F_p[\Phi]}(E^G)^*=\rank_{\F_p[\Phi]}(E^*)_G$. Thus, $\F_p[\Phi]^r \twoheadrightarrow (E^*)_G$. By Nakayama's Lemma, $\F_p[\Gamma]^r \twoheadrightarrow E^*$. Taking duals and using the fact that $\F_p[\Gamma]$ is self-dual, we obtain that $E \cong E^{**} \hookrightarrow \F_p[\Gamma]^r$. Since $\tilde{M}$ is the injective hull of $E$ and $\F_p[\Gamma]$ is an injective module, we obtain that $\tilde{M} \hookrightarrow \F_p[\Gamma]^r$, which implies that $\alpha_S \leq n_S \cdot r$, so it is enough to estimate $r$. To compute this rank $r$, we follow the idea in Section 6 of \cite{hajir2017analytic}. 

On one hand, from the exact sequence
\[
0 \to \mathcal{O}^\times_{L_p(K)}/\mu_p \xrightarrow{p} \mathcal{O}^\times_{L_p(K)} \to \mathcal{O}^\times_{L_p(K)}/p \to 0,
\]
we derive the sequence
\[
(\mathcal{O}^\times_{L_p(K)}/\mu_p)^G \to (\mathcal{O}^\times_{L_p(K)})^G \to (\mathcal{O}^\times_{L_p(K)}/p)^G \to H^1(G, \mathcal{O}^\times_{L_p(K)}/\mu_p). 
\]
We observe that
\begin{itemize}
	\item $(\mathcal{O}^\times_{L_p(K)}/p)^G = E^G$,
	\item $(\mathcal{O}^\times_{L_p(K)})^G \big/ (\mathcal{O}^\times_{L_p(K)}/\mu_p)^G \cong (\mathcal{O}_K^\times)/(\mathcal{O}_K^\times/\mu_p) \cong 
	\mathcal{O}^\times_K/\mathcal{O}^{\times p}_K$,
\end{itemize}
so the sequence becomes
\begin{align}\label{eq1}
	\mathcal{O}^\times_K/\mathcal{O}^{\times p}_K \to E^G \to H^1(G, \mathcal{O}^\times_{L_p(K)}/\mu_p). 
\end{align}

On the other hand, from the exact sequence
\[
0 \to \mu_p \to \mathcal{O}^\times_{L_p(K)} \to \mathcal{O}^\times_{L_p(K)}/\mu_p \to 0,
\]
we get the exact sequence
\begin{align}\label{eq2}
	H^1(G, \mathcal{O}^\times_{L_p(K)}) \to H^1(G, \mathcal{O}^\times_{L_p(K)}/\mu_p) \to H^2(G, \mu_p).
\end{align}
The $p$-group $G$ acts trivially on $\mu_p$, so for $i = 1, 2$, the groups $H^i(G, \mu_p)$ describe the generators and relations of $G$. 

Now, if $j \colon \Cl_K \to \Cl_{L_p(K)}$ is the map induced by the inclusion $K \hookrightarrow L_p(K)$, then $H^1(G, \mathcal{O}^\times_{L_p(K)}) \cong \ker j$ (see 2 in \cite{MR0076803}). 
Moreover, $\ker j$ is equal to the $p$-primary part of $\Cl_K$, which is isomorphic to $\Gal(L_p(K)/K)^{\ab} = G^{\ab}$, by class field theory. From the semisimple version of Dirichlet's unit theorem (Theorem 6.1 in \cite{hajir2017analytic}; for a proof, see Theorem 6.1 in \cite{MR1828251}) and the fact that $F$ is totally imaginary, we obtain that 
\begin{align}\label{eq3}
	\rank_{\F_p[\Phi]}(\mathcal{O}^\times_K/\mathcal{O}^{\times p}_K) \leq \frac{1}{2}[F\colon \Q]. 
\end{align}

From (\ref{eq1}), (\ref{eq2}) and (\ref{eq3}) it follows that
\begin{align*}
	r = \rank_{\F_p[\Phi]}(E^G) &\leq \rank_{\F_p[\Phi]}(\mathcal{O}_K^\times/\mathcal{O}_K^{\times p}) + d_p H^1(G, \mathcal{O}^\times_{L_p(K)}/\mu_p) \\ &\leq \rank_{\F_p[\Phi]}(\mathcal{O}^\times_K/\mathcal{O}_K^{\times p}) + d_p H^1(G, \mathcal{O}^\times_{L_p(K)}) + d_p H^2(G, \mu_p) \\
	&\leq \frac{1}{2}[F\colon \Q] + d(G) + r(G),
\end{align*}
where $d_p$ is the usual $p$-rank, and $d(G)$ and $r(G)$ are the number of generators and relations of $G$. Therefore:
\begin{align*}
	\beta_S &= [F\colon \Q]\cdot n_S - \alpha_S\\
&\geq [F\colon \Q]\cdot n_S - n_S \cdot r \\
&\geq [F\colon \Q]\cdot n_S - n_S \cdot \left(\frac{1}{2}[F\colon \Q] + d(G) + r(G)\right) \\
&\geq n_S \left(\frac{1}{2}[F\colon \Q] - d(G) - r(G)\right). 
\end{align*}
We can thus choose $t \geq (\frac{1}{2}[F\colon \Q] - d(G) - r(G))$ such that $Q := \oplus P_S^{n_S \cdot t}$ is isomorphic to a submodule of $Q_1$. This new module $Q$ is a free $\F_p[\Gamma]$-module of rank $t$. Moreover, it is injective, so $P = Q_1/Q$ is a projective $\F_p[\Gamma]$-module with $Q_1 = Q \oplus P$. Let $M = M_1 \oplus P$. Then 
\[
R(L_p(K)) = M_1 \oplus N \oplus Q_1 \cong M_1 \oplus N \oplus Q \oplus P \cong M \oplus N \oplus Q,
\]
with $E \subset M$ and $\rank_{\F_p[\Gamma]} Q \geq \frac{1}{2}[F\colon \Q] - d(G)-r(G)$. 

Since $M$ is a stably free $\F_p[\Gamma]$-module, we can conclude that it is a free $\F_p[\Gamma]$-module, so the proof is complete. 
\end{proof}

The only thing left to show is the existence of a prime $\Lambda$ of $L_p(K)$ with properties (\ref{cond1}) and (\ref{cond2}). The proof follows the steps of Proposition 1 in the first version of \cite{MR2772089}. Let $M$ and $Q =  \bigoplus_{i=1}^t \F_p[\Gamma]q_i$ be the $\F_p[\Gamma]$-submodules of $R(L_p(K))$ given by Lemma \ref{lemma2}. Then, by assumption,
\[
t \geq \frac{1}{2}[F\colon\Q] - d(G)-r(G) \geq d(G) + d(\Phi) \geq d(\Gamma).
\]
Let $\{ \sigma_1, \dots, \sigma_d\}$ be a system of minimal generators for $\Gamma$, $d = d(\Gamma)$. Let $u = \sum_{i=1}^d (\sigma_i - 1)q_i \in Q \subset R(L_p(K))$. By Lemma \ref{lemma3} for $M = L_p(K)$, there exists $\Lambda\mathcal{O}_{L_p(K)}$ a prime of degree $1$, not lying over $p$, such that $u = \left(\Lambda \mod{U(L_p(K))^p} \right)$. Assume that there exist $\varepsilon \in \mathcal{O}_{L_p(K)}^\times$ and $a_{\sigma} \in \F_p$ such that $\displaystyle(\varepsilon \mod{U^\prime(L_p(K))}) + \sum_{\sigma \in \Gamma} a_\sigma(\sigma \Lambda \mod{U^\prime(L_p(K))}) =0$. Observe that:
\begin{itemize}
	\item $\displaystyle \varepsilon \mod{U(L_p(K))^p} + \sum_{\sigma \in \Gamma} a_\sigma(\sigma \Lambda \mod{U(L_p(K))^p})  \in N$;
	\item $\displaystyle \sum_{\sigma \in \Gamma} a_\sigma(\sigma \Lambda \mod{U(L_p(K))^p}) = \sum_{\sigma \in \Gamma} a_\sigma(\sigma u) \in Q$;
	\item $\varepsilon \mod{U(L_p(K))^p} \in E \subset M$.
\end{itemize}
By Lemma \ref{lemma2}, it follows that $\varepsilon \mod{U(L_p(K))^p} = \sum_{\sigma \in \Gamma} a_\sigma(\sigma \Lambda \mod{U(L_p(K))^p})= 0$. On one hand, since $U(L_p(K))^p \cap \mathcal{O}^\times_{L_p(K)} = \mathcal{O}^{\times p}_{L_p(K)}$, it follows that $\varepsilon \in \mathcal{O}^{\times p}_{L_p(K)}$. On the other hand, $\sum_{\sigma \in \Gamma} a_\sigma(\sigma \Lambda \mod{U(L_p(K))^p}) = 0$ implies $\sum_{\sigma \in \Gamma} a_\sigma \sigma [\sum_{i=1}^d (\sigma_i - 1) q_i] = \sum_{\sigma \in \Gamma} a_\sigma(\sigma u) = 0$. This implies $\sum_{\sigma \in \Gamma} a_\sigma \sigma (\sigma_i - 1) = 0$, for all $1 \leq i \leq d$, so $\sum_{\sigma \in \Gamma} a_\sigma \sigma (\tau - 1) = 0$, for all $\tau \in \Gamma$, meaning that $a_\sigma$ must be constant for all $\sigma \in \Gamma$, i.e. $a_\sigma = a \in \F_p$, for some $a \in \F_p$. We have thus shown that $\Lambda$ has properties (\ref{cond1}) and (\ref{cond2}), so the proof is complete. 

\section{Embedding Problem}
In this section, we introduce some results about the embedding problem, used in the proof of Proposition \ref{prop2}. A detailed exposition of this can be found in \cite{MR0337894}. 

Let $F$ be a number field and let $G_F$ be the absolute Galois group of $F$. Let $K/F$ be a finite Galois extension with Galois group $G$. For an extension of finite groups $(\varepsilon) \colon 1 \to A \to E \to G \to 1$, the embedding problem $(G_F, \varepsilon)$ is defined by the diagram 
\begin{center}
		\begin{tikzcd}
		 &&&&G_F \ar[d, "\varphi"] \\
	&1 \ar[r] & A \ar[r] & E \ar[r, "\pi"] & G \ar[r] &1,
\end{tikzcd}
	\end{center}
	where $\varphi$ is the canonical surjection. A continuous homomorphism $\psi \colon G_F \to E$ is called a solution of $(G_F, \varepsilon)$ if it satisfies the condition $\pi \circ \psi = \varphi$. A solution $\psi$ is called a proper solution if it is surjective. If $(\varepsilon)$ is a nonsplit extension, then every solution of the embedding problem is a proper solution (Satz 2.3 in \cite{MR0244190}). We are only interested in the case when the extension is nonsplit, so we can assume that if a solution to the embedding problem exists, then it is proper. This is the same as finding an extension $M/F$ containing $K/F$ such that $\Gal(M/F)\cong E$ compatibly with $\Gal(K/F)=G$. 
	When such a solution exists, we say that $(G_F, \varepsilon)$ is solvable.  

	For each prime $\mathfrak{p}$ of $F$, we denote by $F_\mathfrak{p}$ (resp. $K_\mathfrak{p}$) the completion of $F$ at $\mathfrak{p}$ (resp. of $K$ at a prime above $\mathfrak{p}$). Let $G_{F_\mathfrak{p}}$ be the absolute Galois group of $F_\mathfrak{p}$, $G_\mathfrak{p} = \varphi(G_{F_\mathfrak{p}}) \subset G$ (which is isomorphic to the decomposition group of $\mathfrak{p}$ in $\Gal(K/F)$) and $E_\mathfrak{p} = \pi^{-1}(G_\mathfrak{p}) \subset E$. Then the local embedding problem $(G_{F_\mathfrak{p}}, \varepsilon_\mathfrak{p})$ is defined by 
	\begin{center}
		\begin{tikzcd}
		&&&&G_{F_\mathfrak{p}} \ar[d, "\varphi_\mathfrak{p}"] \\
&1 \ar[r] & A \ar[r] & E_\mathfrak{p} \ar[r, "\pi_\mathfrak{p}"] & G_\mathfrak{p} \ar[r] &1,
\end{tikzcd}
	\end{center}

	We have the following results from \cite{MR0337894} (Satz 2.2, Satz 4.7, Satz 5.1).
	\begin{theorem}\label{satz2.2} 
		Let $(G_F, \varepsilon)$ be an embedding problem with abelian kernel $A$. If the map 
		\begin{align*}
			H^2(G_F, A) \to \prod_{\mathfrak{p} \in P} H^2(G_{F_\mathfrak{p}}, A)
		\end{align*}
		is injective, then the embedding problem $(G_F, \varepsilon)$ has a solution if and only if the local embedding problems $(G_{F_\mathfrak{p}}, \varepsilon_\mathfrak{p})$ have solutions, for all $\mathfrak{p} \in P$. Here $P$ is the set of primes of $F$. 
	\end{theorem}
	\begin{theorem}\label{satz4.7}
		If $A$ is a trivial finite $G$-module (i.e. $A = \Z/n\Z$) or $A$ is the dual of one (i.e. $A=\mu_n$), then all maps 
		\begin{align*}
			H^q (F, A) \to \prod_{\mathfrak{p}} H^q(F_\mathfrak{p}, A), \quad q \geq 0,
		\end{align*}
		are injective. Here we have $H^q (F, A) = H^q(G_F, A)$. 
	\end{theorem}
	\begin{theorem}\label{satz5.1}
		If $K_\mathfrak{p}/F_\mathfrak{p}$ is a cyclic extension of local fields, then the following conditions are equivalent.
		\begin{enumerate}[label = (\roman*)]
			\item Every embedding problem corresponding to the extension $K_\mathfrak{p}/F_\mathfrak{p}$ with an arbitrary (not necessarily abelian) kernel $A$ of exponent $n$ is solvable. 
			\item Every n-th root of unity in $F_\mathfrak{p}$ is the norm of an element of $K_\mathfrak{p}$. 
		\end{enumerate}

		This is always true if $K_\mathfrak{p}/F_{\mathfrak{p}}$ is unramified. 

		If $K_\mathfrak{p}/F_\mathfrak{p}$ is tamely ramified with ramification index $e$, then $(i)$ and $(ii)$ are true if and only if $n^\prime e \mid (q-1)$, where $n^\prime = \prod_{p \mid e} p^{v_p(n)}$ and $q$ is the number of elements of the residue field of $F_\mathfrak{p}$. 
	\end{theorem}

Going back to our case, consider the following embedding problem:
\begin{center}
		\begin{tikzcd}
		&&&&G_{F} \ar[d] \\
&1 \ar[r] & V \ar[r] & \Gamma^\prime \ar[r] & \Gamma \ar[r] &1,
\end{tikzcd}
	\end{center}
with $\Gal(L_p(K)/F) = \Gamma = G \rtimes \Phi$, $\Gal(K/F) = \Phi$, and $K/F$ satisfies property \textbf{P} (which implies that $L_p(K)/F$ also satisfies this property). By Theorem \ref{satz5.1}, all the local embedding problems have solutions, so in order to use Theorem \ref{satz2.2}, we need to prove that the map 
\begin{align*}
			H^2(G_F, V) \to \prod_{\mathfrak{p} \in P} H^2(G_{F_\mathfrak{p}}, V)
\end{align*}
is injective. Note that $V$ is not a trivial $\F_p[\Gamma]$-module, so we can't apply Theorem \ref{satz4.7} directly. Consider the following commutative diagram

\[
    \begin{tikzcd}
    H^2(G_F, V)\arrow{r}  \arrow{d}{\text{res}}
    &\displaystyle \prod_{\mathfrak{p}} H^2(G_{F_\mathfrak{p}}, V) \arrow{d}{\text{res}} \\
    H^2(G_K, V)\arrow{r}
    & \displaystyle \prod_{\mathfrak{p}} H^2(G_{K_\mathfrak{p}}, V)
\end{tikzcd}
\]
Using spectral sequences, we can prove that $H^2(G_K, V)^\Phi \cong H^2(G_F, V)$, so the map on the left is injective. Similarly, it can be proved that the map on the right is injective. Since $K = \overline{F}^{\ker \overline{\rho}}$, where $\overline{\rho} \colon G_F \to \GL_2(\F_p)$ is irreducible and has image $\Phi$, the adjoint action coming from $G_F$ is killed on $G_K$, so 
\[
H^2(G_K, V) \cong H^2(G_K, \F_p^n),
\]
where $n = \dim_{\F_p} V$. 
So by Theorem \ref{satz4.7}, the map on the bottom is injective. It thus follows that the top map is injective, and so Theorem \ref{satz2.2} applies. 

On one hand, if the group extension is split, we claim that a solution to the embedding problem is given by $M = L_p(K)(\sqrt[p]{a_1}, \dots, \sqrt[p]{a_n})/F$, with $a_1, \dots, a_n \in K$, and $\Gal(M/F) \cong \Gamma^\prime \cong V \rtimes \Gamma$. On the other hand, we claim that any two solutions differ by a split extension. To summarize: 

\begin{proposition}\label{embedding}
	Let $K/F$ be an extension with Galois group $\Phi$ that satisfies property \textbf{P}. Consider the extension $L_p(K)/F$ with Galois group $\Gamma$. The embedding problem 
	\[
	1 \to V \to \Gamma^\prime \to \Gamma \to 1
	\]
	always has a solution. Furthermore, if $L_p(K)(\sqrt[p]{\alpha_1}, \dots, \sqrt[p]{\alpha_n})/F$ is a solution, with $\alpha_i \in L_p(K)^\times/L_p(K)^{\times p}$, then all the other solutions are given by $L_p(K)(\sqrt[p]{\alpha_1 a_1}, \dots, \sqrt[p]{\alpha_n a_n})/F$, where $a_i \in K^\times/K^{\times p}$, $\alpha_i a_i \neq 0$ in $L_p(K)^\times/L_p(K)^{\times p}$, and $\Gal(K(\sqrt[p]{a_1}, \dots, \sqrt[p]{a_n})/F) \cong V \rtimes \Phi$. 
\end{proposition}

\section{Proof of Proposition \ref{prop2}} \label{proofofprop2}
In this section, we will provide a proof for Proposition \ref{prop2}. The proof can be split into two cases: when the following sequence splits or when it doesn't split: 
\[
1 \to V \to \Gamma^\prime \to \Gamma \to 1. 
\]
The case when the extension does not split will use the embedding problem combined with the case when the extension splits, and will be treated at the end of this section. Assume first that the sequence splits. In this case, $\Gamma^\prime \cong V \rtimes \Gamma$, and we can work over $K$. 

Let $n=\dim_{\F_p} V$ and let $m = |\Phi|$. Let $g_1, \dots, g_n$ be generators of the action of $\Phi$ on $V$. Let $T = \frac{(p^{2n}-1)(p^{2n}-p)}{(p^2-1)(p^2-p)}$. 
We can assume that $[F\colon \Q] \geq 2d (T+2)$, where $d = d(\Phi)$ is the number of generators of $\Phi$, by replacing $F$ with some finite extension of $F$ given by Proposition \ref{prop1}. Recall that Proposition \ref{prop1} constructs a new extension that satisfies property \textbf{P} if the initial extension satisfied this property. Note that $R(L_p(K))^G \cong R(K)$ and $R^\prime(L_p(K))^G \cong R^\prime(K)$ as $\F_p[\Phi]$-modules, $N^G \cong \ker(R(K) \to R^\prime(K))$, and $(\mathcal{O}_K^\times \otimes \F_p) \cap Q^G = 0$, where $Q$ is the $\F_p[\Gamma]$-module obtained in Lemma \ref{lemma2}. Let $\{ \sigma_1, \dots, \sigma_d\}$ be a generator system of $\Phi$. The free $\F_p[\Phi]$-module $Q^G$ can be written as $Q^G =\bigoplus_{i=1}^t \F_p[\Phi]q_i$, with $t \geq d(T+2)$. 

Using Lemma \ref{lemma3} applied to $M=K$, we obtain primes $\lambda_i \mathcal{O}_K$ that are completely split in $L_p(K)/K$, and satisfy $N(\lambda_i) \equiv 1 \pmod{p}$ and 
\[
\lambda_i \mod{U(K)^p} = \sum_{j=1}^d (1+\sigma_j)q_{(i-1)d+j}
\] 
in $R(K)$, for $1 \leq i \leq T$. Moreover, we can also assume that the primes of $F$ below $\lambda_i$ split completely inn $K/F$

Since $\Phi$ has order prime to $p$, the $\F_p[\Phi]$-module $R(K)$ can be decomposed as a direct sum of isotypic components: 
\[ R(K) = \bigoplus_W W^{n_W},\] where each $W$ is a simple $\F_p[\Phi]$-representation.  Each $a \in R(K)$ may therefore be written as $a = \sum a_W$, where $a_W \in W^{n_W}$. The isotypic projection $P_W$ of $R(K)$ onto $W^{n_W}$ is given by the formula 
\[
P_W = \frac{\dim W}{\mid\Phi\mid \cdot \dim_{\F_p}\End W} \sum_{g \in \Phi} \chi_W(g^{-1}) g,
\]
where $\chi_W$ is the character of $W$. This is a modified version of Theorem 8 in \cite{MR0450380} that works for (not necessarily algebraically closed) finite fields. Note that $P_W (P_U (a)) = a_W$ if $U \cong W$ and $P_W (P_U (a)) = 0$ if $U \not\cong W$. To ease notation, we will write $n_{g, W} = \frac{\dim W \cdot \chi_W(g^{-1})}{\mid\Phi\mid \cdot \dim_{\F_p}\End W}$. We can then write $P_W (a) = \displaystyle \sum_{g \in \Phi} n_{g, W} g(a)$. Note that since $V$ is an absolutely irreducible $\F_p[\Phi]$-representation, $n_{g, V}$ is well-defined and nonzero in $\F_p$. 

Let $(a_1, \dots, a_n, b_1, \dots, b_n)$ and $(x_1, \dots, x_n, y_1, \dots, y_n)$ be two nonzero elements of $(\Z/p\Z)^{2n}$ that are not multiples of each other. This pair generates a $(\Z/p\Z)^2$-subgroup of $(\Z/p\Z)^{2n}$. There are $\frac{(p^{2n}-p)(p^{2n}-1)}{(p^2-p)(p^2-1)}$ subspaces spanned by such a pair, i.e. $(\Z/p\Z)^2$-subgroups; label them from $1$ to $\frac{(p^{2n}-p)(p^{2n}-1)}{(p^2-p)(p^2-1)}$. For each such pair, choose a prime $\lambda_\ell$ from the ones constructed above (we can do this since $T = \frac{(p^{2n}-p)(p^{2n}-1)}{(p^2-p)(p^2-1)}$) and define $\nu_1$ and $\nu_2$ to be products of conjugates of $\lambda_\ell$, with $1 \leq \ell \leq T$, with the properties that the exponents of $g_1^{-1}(\lambda_\ell), g_2^{-1}(\lambda_\ell), \dots g_n^{-1}(\lambda_\ell)$ in $\nu_1$ are $a_1, \dots, a_n$, and the exponents of $g_1^{-1}(\lambda_\ell), g_2^{-1}(\lambda_\ell), \dots g_n^{-1}(\lambda_\ell)$ in $\nu_2$ are $b_1, \dots, b_n$, respectively. Write $a_{i, \ell} = a_i$, $b_{i, \ell} = b_i$, $x_{i, \ell} = x_i$, and $y_{i, \ell} = y_i$. Let $\displaystyle \nu_1 = \prod_{i=1}^T \prod_{g \in \Phi} g(\lambda_i)^{s_{g,i}}$ and $\displaystyle \nu_2 = \prod_{i=1}^T \prod_{g \in \Phi} g(\lambda_i)^{t_{g,i}}$. Write $g_i(\nu_1) = \lambda_\ell^{a_{i, \ell}} \omega_{i, \ell}$ and $g_i(\nu_2) = \lambda_\ell^{b_{i, \ell}} \xi_{i, \ell}$, for all $1\leq i\leq n$, with $\omega_{i, \ell}$ and $\xi_{i, \ell}$ not divisible by $\lambda_\ell$. For each $\ell$, let $\overbar{r_\ell}$ be a non $p$-power modulo $\lambda_\ell$. Lift this $\overbar{r_\ell}$ to a principal ideal $\kappa_\ell$ in $\mathcal{O}_K$. We distinguish two cases, each containing two subcases: 
\begin{enumerate}
	\item If $a_{i, \ell} \neq 0$, for some $i$, let 
	\begin{align*}
		c_{j, \ell} &= \left\{
		\begin{array}{ll}
			x_{j, \ell} a_{i, \ell} - x_{i, \ell} a_{j, \ell}, &\text{ if } j \neq i \\
			0, &\text{ if } j = i,
		\end{array}
		\right.\\
		d_{j, \ell} &= y_{j, \ell} a_{i,\ell} - x_{i, \ell}b_{j, \ell}, \text{ for all } j.  
	\end{align*}
	Since $(a_{1, \ell}, \dots, a_{n, \ell}, b_{1, \ell}, \dots, b_{n, \ell})$ and $(x_{1, \ell}, \dots, x_{n, \ell}, y_{1, \ell}, \dots, y_{n, \ell})$ are not multiples of each other, at least one of $c_{j, \ell}$ and $d_{j, \ell}$ is nonzero. 
	\begin{enumerate}
		\item If $c_{j, \ell} \neq 0$, for some $j \neq i$, let 
		\begin{align*}
			A_{k, \ell} &= \left\{
			\begin{array}{ll}
				\omega_{k, \ell}^{-1} &\text{ if } k = i\\
				\kappa_{\ell} \cdot \omega_{k, \ell}^{-1} &\text{ if } k=j\\
				\kappa_\ell^{c_{k, \ell}c_{j, \ell}^{-1}}\cdot \omega_{k, \ell}^{-1} &\text{ otherwise.}\\
			\end{array}
			\right.\\
			B_{k, \ell} &= \kappa_\ell^{d_{k, \ell}c_{j, \ell}^{-1}} \cdot \xi_{k, \ell}^{-1}, \text{ for all } k. 
		\end{align*}

		\item If $d_{j, \ell} \neq 0$, for some $j$ (possibly $j = i$), let 
		\begin{align*}
		 	A_{k, \ell} &= \left\{
			\begin{array}{ll}
				\omega_{k, \ell}^{-1} &\text{ if } k = i\\
				\kappa_\ell^{c_{k, \ell}d_{j, \ell}^{-1}}\cdot \omega_{k, \ell}^{-1} &\text{ otherwise.}\\
			\end{array}
			\right.\\
			B_{k, \ell} &= \left\{
			\begin{array}{ll}
				\kappa_\ell \cdot \xi_{k, \ell}^{-1} &\text{ if } k = j\\
				\kappa_\ell^{d_{k, \ell}d_{j, \ell}^{-1}}\cdot \xi_{k, \ell}^{-1} &\text{ otherwise.}\\
			\end{array}
			\right.
		 \end{align*} 
	\end{enumerate}
	\item If $b_{i, \ell} \neq 0$, for some $i$, let 
	\begin{align*}
		c_{j, \ell} &= x_{j, \ell} b_{i, \ell} - y_{i, \ell} a_{j, \ell}, \text{ for all} \\
		d_{j, \ell} &= \left\{
		\begin{array}{ll}
			y_{j, \ell} b_{i,\ell} - y_{i, \ell}b_{j, \ell}, &\text{ for } j \neq i \\
			0, &\text{ if } j = i.
		\end{array}
		\right.\\
	\end{align*}
	Since $(a_{1, \ell}, \dots, a_{n, \ell}, b_{1, \ell}, \dots, b_{n, \ell})$ and $(x_{1, \ell}, \dots, x_{n, \ell}, y_{1, \ell}, \dots, y_{n, \ell})$ are not multiples of each other, at least one of $c_{j, \ell}$ and $d_{j, \ell}$ is nonzero. 
	\begin{enumerate}
		\item If $d_{j, \ell} \neq 0$, for some $j \neq i$, let 
		\begin{align*}
			A_{k, \ell} &= \kappa_\ell^{c_{k, \ell}d_{j, \ell}^{-1}} \cdot \omega_{k, \ell}^{-1}, \text{ for all } k.\\
			B_{k, \ell} &= \left\{
			\begin{array}{ll}
				\xi_{k, \ell}^{-1} &\text{ if } k = i\\
				\kappa_{\ell} \cdot \xi_{k, \ell}^{-1} &\text{ if } k=j\\
				\kappa_\ell^{d_{k, \ell}d_{j, \ell}^{-1}}\cdot \xi_{k, \ell}^{-1} &\text{ otherwise.}\\
			\end{array}
			\right.\\
		\end{align*}

		\item If $c_{j, \ell} \neq 0$, for some $j$ (possibly $j = i$), let 
		\begin{align*}
		 	A_{k, \ell} &= \left\{
			\begin{array}{ll}
				\kappa_\ell\cdot \omega_{k, \ell}^{-1} &\text{ if } k = j\\
				\kappa_\ell^{c_{k, \ell}c_{j, \ell}^{-1}}\cdot \omega_{k, \ell}^{-1} &\text{ otherwise.}\\
			\end{array}
			\right.\\
			B_{k, \ell} &= \left\{
			\begin{array}{ll}
				\xi_{k, \ell}^{-1} &\text{ if } k = i\\
				\kappa_\ell^{d_{k, \ell}c_{j, \ell}^{-1}}\cdot \xi_{k, \ell}^{-1} &\text{ otherwise.}\\
			\end{array}
			\right.
		 \end{align*} 
	\end{enumerate}
\end{enumerate}

Let $\displaystyle R = \prod_{i, \ell} g_i^{-1}(\lambda_\ell)$. We would like to construct a prime $\alpha$ of $\mathcal{O}_K$ with $\alpha \equiv g_i^{-1}(A_{i, \ell}) \pmod{g_i^{-1}(\lambda_\ell)}$, for all $i$ and $\ell$. Since the ideals $g_i^{-1}(\lambda_\ell) \mathcal{O}_K$ are pairwise coprime, the Chinese remainder theorem tells us that constructing such an $\alpha$ is equivalent to constructing an $\alpha$ that satisfies a specific congruence modulo $R\mathcal{O}_K$, say $\alpha \equiv r_1 \pmod{R}$, compatible with $\alpha \equiv g_i^{-1}(A_{i, \ell}) \pmod{g_i^{-1}(\lambda_\ell)}$. Similarly, to construct a prime $\beta$ of $\mathcal{O}_K$ with $\beta \equiv g_i^{-1}(B_{i, \ell}) \pmod{g_i^{-1}(\lambda_\ell)}$, it is enough to construct a prime $\beta$ with a specific compatible congruence modulo $R\mathcal{O}_K$, say $\beta \equiv r_2 \pmod{R}$. Use the existence theorem of class field theory to construct an abelian extension $\tilde{M}/K$ whose Galois group is in a natural correspondence with the ray classes modulo $R\mathcal{O}_K$. Use Lemma \ref{lemma3} with $M = K$ again to find two primes $\alpha$ and $\beta$ of $\mathcal{O}_K$ that split completely in $L_p(K)/K$, prime to $p$, such that 
\[
\alpha \mod{U(K)^p} = -P_V (\nu_1) + \sum_{W \neq V} P_W \left(\sum_{j=1}^d (1+\sigma_j)q_{Td+j}\right)
\]
and
\[
\beta \mod{U(K)^p} = -P_V (\nu_2) + \sum_{W \neq V} P_W \left(\sum_{j=1}^d (1+\sigma_j)q_{(T+1)d+j} \right)
\]
in $R(K)$, and
\begin{align*}
	\alpha &\equiv r_1 \pmod{R}\\
	\beta &\equiv r_2 \pmod{R}. 
\end{align*}

Observe that while taking projections $P_V(\nu_1)$ and $P_V(\nu_2)$ modifies the exponents of the primes $g(\lambda_\ell)$ in $\nu_1$ and $\nu_2$, those exponents still appear as the exponents of some other primes. In other words, if $(a_1, \dots, a_n, b_1, \dots, b_n)$ generates a $\Z/p\Z$ subgroup of $(\Z/p\Z)^{2n}$, then there is a prime $g_i(\lambda_\ell)$ such that the the exponents of $g_1^{-1}(g_i(\lambda_\ell)), g_2^{-1}(g_i(\lambda_\ell)), \dots, g_n^{-1}(g_i(\lambda_\ell))$ in $\nu_1$ are $a_1, \dots, a_n$ and the exponents of $g_1^{-1}(g_i(\lambda_\ell)), g_2^{-1}(g_i(\lambda_\ell)), \dots, g_n^{-1}(g_i(\lambda_\ell))$ in $\nu_2$ are $b_1, \dots, b_n$.

Note that $K^\times/K^{\times p}$ is an $\F_p[\Phi]$-module, so we can consider the projections to the $V$-eigenspace of $\nu_1 \alpha$ and $\nu_2\beta$. Construct $K_1$ to be the Galois closure of $K(\sqrt[p]{P_V (\nu_1 \alpha)})$ over F, and $K_2$ to be the Galois closure of $K(\sqrt[p]{P_V (\nu_2 \beta)})$ over F. Then $\Gal(K_1/F) \cong \Gal(K_2/F) \cong V \rtimes \Phi$. Moreover, since $P_V(\nu_1 \alpha) = 0$ in $R(K)$, every prime lying over $p$ in $K$ splits completely in $K_1/K$. The same holds true for $K_2/K$. Consider their compositum $\tilde{K} = K_1.K_2$ and let $\tilde{L} = \tilde{K}.L_p(K)$. We claim that:
\begin{itemize}
	\item The extension $\tilde{L}/L_p(K)$ is unramified at $p$.
	\item Recall that $g(\alpha)$, $g(\beta)$, $g(\lambda_i)$ split completely in $L_p(K)/K$ by construction, for all $g \in \Phi$ and $1 \leq i \leq T$. Let $\tilde{S}$ be the set of primes of $L_p(K)$ lying above these primes. Then $\tilde{L}/L_p(K)$ is the maximal elementary abelian $p$-extension of $L_p(K)$ which is unramified outside $\tilde{S}$.
	\item All the $(\Z/p\Z)^2$-subgroups of $\Gal(\tilde{L}/L_p(K))$ appear as decomposition groups of some prime in $L_p(K)$ lying over $g(\lambda_i)$, for $g \in \Phi$ and $1 \leq i \leq T$. 
\end{itemize}

Consider $P_V (\nu_1 \alpha)$ in $R(K)$. By construction, $P_V(\nu_1\alpha) = 0$ in $R(K)$, so it must remain trivial in $R^\prime(K)$. Similarly, $P_V (\nu_2 \beta) = 0$ in $R^\prime(K)$. It implies that $K_1/K$ and $K_2/K$ are unramified at $p$, and thus $\tilde{K}/K$ is unramified at $p$. Combine this with the fact that $L_p(K)/K$ is unramified everywhere to conclude that $\tilde{L} = \tilde{K}.L_p(K)/L_p(K)$ is unramified at $p$.

To show that $\tilde{L}$ is the maximal elementary abelian $p$-extension of $L_p(K)$ unramified outside $\tilde{S}$, it is enough to show that $\tilde{K}$ is the maximal elementary abelian $p$-extension of $K$ unramified outside $S = \{ g(\alpha), g(\beta), g(\lambda_i) \mid g \in \Phi, 1 \leq i \leq T \}$. This is true because $\tilde{S}$ represents the set of primes of $L_p(K)$ lying over the primes in $S$, and all the primes in $S$ split completely in $L_p(K)/K$.  

To this end, consider an elementary abelian $p$-extension of $K$ unramified outside $S$, call it $K(\sqrt[p]{\gamma})/K$. Then
\[
\gamma \equiv \eta \cdot \prod_{i=1}^T \prod_{g \in \Phi} g(\lambda_i)^{c_{g,i}} \prod_{g \in \Phi} g(\alpha)^{a_g} \prod_{g \in \Phi} g(\beta)^{b_g} \mod{K^{\times p}},
\]
for some $\eta \in \mathcal{O}_K^\times$ and $a_g, b_g, c_{g, i} \in \Z$, for $1 \leq i \leq T$ and $g \in \Phi$. Since $K(\sqrt[p]{\gamma})/K$ is unramified at $p$, it follows that $\gamma = 0$ in $R^\prime(K)$, so $\gamma \in N^G$. Thus $P_W(\gamma) = 0 \in R(K)$, for all $W$, which in turn means that $P_W(\gamma) = 0 \in R^\prime(K)$. From construction, $g(\lambda_i), g(\alpha), g(\beta) \in Q^G \subset R^\prime(K)$. Moreover, since $\eta \in \mathcal{O}_K^\times \otimes \F_p$, and $\mathcal{O}_K^\times \otimes \F_p$ intersects $Q^G$ and $N^G$ trivially, it follows that $\eta = 0$ in $R^\prime(K)$, so $\eta \in \mathcal{O}_K^{\times p}$, since $\eta$ is a global unit. 

On one hand, consider $P_W(\gamma)$, for $W \neq V$:
\begin{align*}
	0 = P_W(\gamma) &= \sum_{i=1}^T\sum_{g\in \Phi} \left( \sum_{h\in  \Phi} c_{gh^{-1}, i} \cdot n_{h, W} g(\lambda_i)\right) \\ &+ \sum_{g\in \Phi} \left( \sum_{h \in \Phi} a_{gh^{-1}} \cdot n_{h, W}  g(\alpha) \right) \\ &+ \sum_{g\in \Phi} \left( \sum_{h\in\Phi} b_{gh^{-1}} \cdot n_{h, W} g(\beta)\right) 
\end{align*}
By construction, $\lambda_i, \alpha, \beta$ and their conjugates are all linearly independent in $Q^G$, so it follows that their coefficients have to be $0$:

\begin{align*}
	\sum_{h \in \Phi} n_{h, W} \cdot c_{gh^{-1}, i} &= 0, \\
	\sum_{h \in \Phi} n_{h, W} \cdot a_{gh^{-1}} &= 0,\\
	\sum_{h \in \Phi} n_{h, W}\cdot b_{gh^{-1}} &= 0,
\end{align*}
for all $1 \leq i \leq T$ and $g \in \Phi$. 

On the other hand, recall that $P_V(\nu_1\alpha) = 0$ and $P_V(\nu_2 \beta) =0$ in $R(K)$.  Consider $P_V(\gamma)$: 
\begin{align*}
	0 = P_V(\gamma) = \sum_{i=1}^T \sum_{g \in \Phi} \left( \sum_{h\in \Phi} n_{h, V}\left(c_{gh^{-1}, i} - \sum_{\tau \in \Phi} (b_{g\tau^{-1}}y_{\tau h^{-1}, i} + a_{g\tau^{-1}}x_{\tau h^{-1}, i})\right)\right) g(\lambda_i).
\end{align*}
Using the same argument as above, we observe that 
\[
\sum_{h\in \Phi} n_{h, V}(c_{gh^{-1}, i} - \sum_{\tau\in \Phi} (b_{g\tau^{-1}}y_{\tau h^{-1}, i} + a_{g\tau^{-1}}x_{\tau h^{-1}, i})) = 0,
\]
for all $1 \leq i \leq T$ and all $g\in \Phi$. 

Finally, putting these things together, we obtain that
\[
\gamma \equiv \prod_{g \in \Phi} g(P_V (\nu_1 \alpha))^{a_g} \prod_{g\in \Phi} g(P_V (\nu_2 \beta))^{b_g} \mod{K^{\times p}},
\]
meaning that $K(\sqrt[p]{\gamma}) \subset \tilde{K}$, which proves the maximality of $\tilde{K}$, and concludes the proof for the second claim.  

Just as above, we observe that it is enough to prove the third claim for $\tilde{K}/K$, the statement for $\tilde{L}/L_p(K)$ following from it, since $g(\lambda_i)$ splits completely in $L_p(K)/K$, for all $g\in \Phi$ and $1 \leq i \leq T$, and $\Gal(\tilde{L}/L_p(K)) \cong \Gal(\tilde{K}/K)$. The following short lemma proves this statement for $\tilde{K}/K$. 
\begin{lemma}
	All the $(\Z/p\Z)^2$-subgroups of $\Gal(\tilde{K}/K)$ appear as decomposition groups of some $g_i(\lambda_\ell)$, for $1 \leq \ell \leq T$ and $i = i(\ell)$. 
\end{lemma}
\begin{proof}
	Take a $(\Z/p\Z)^2$-subgroup $N$ of $\Gal(\tilde{K}/K)$; note that there are $T = \frac{(p^{2n}-1)(p^{2n}-p)}{(p^2-1)(p^2-p)}$ of them. Assume that $N$ is generated by $(a_1, \dots, a_n, b_1, \dots, b_n)$ and $(x_1, \dots, x_n, y_1, \dots, y_n)$ with $(a_1, \dots, a_n, b_1, \dots, b_n) \neq 0$, $(x_1, \dots, x_n, y_1, \dots, y_n) \neq 0$, and $(a_1, \dots, a_n, b_1, \dots, b_n)$ and $(x_1, \dots, x_n, y_1, \dots, y_n)$ are not multiples of each other. From the above discussion, we know that there is a prime $\lambda = g_i(\lambda_\ell)$ such that the exponents of $g_1^{-1}(\lambda), g_2^{-1}(\lambda), \dots, g_n^{-1}(\lambda)$ are $a_1, \dots, a_n$ in $\nu_1$, and $b_1, \dots, b_n$ in $\nu_2$, respectively. This implies that the inertia subgroup corresponding to $\lambda$ in $\Gal(\tilde{K}/K)$ is generated by $(a_1, \dots, a_n, b_1, \dots, b_n)$. Moreover, the prime $\lambda$ satisfies
	\begin{align*}
		\alpha &\equiv g_j^{-1}(A_{j, \ell}) \pmod{g_j^{-1}(\lambda)}, \\
		\beta &\equiv g_j^{-1}(B_{j, \ell}) \pmod{g_j^{-1}(\lambda)},
	\end{align*}
	for all $1 \leq j \leq n$, which implies that the fixed field of inertia is nontrivial, but the fixed field of $N$ is trivial. Thus, the decomposition group of $\lambda$ is isomorphic to $N$. This concludes the proof, since $N$ was chosen arbitrary. 
\end{proof}

We claim that the extension $\tilde{L}/L_p(K)$ satisfies the conditions of Lemma \ref{lemma1}, with $\tilde{S}$ instead of S. We know that $L_p(L_p(K)) = L_p(K)$ and we have just proved that $\tilde{L}$ is the maximal elementary abelian $p$-extension of $L_p(K)$ unramified outside $\tilde{S}$, so the only thing left to prove is that the following map is surjective
\begin{align*}
	\bigoplus_v H_2(D_v, \Z) \to H_2(\Gal(\tilde{L}/L_p(K)), \Z),
\end{align*}
where the sum is over all the primes of $L_p(K)$. The key fact is that for a finite abelian group $G$, the homology group $H_2(G,\Z)$ is isomorphic to the second exterior power of $G$, $\displaystyle\bigwedge^2 (G)$ (Lemma 5 in \cite{MR0466073}). Recall that $\Gal(\tilde{L}/L_p(K)) \cong \Gal(\tilde{K}/K) \cong (\Z/p\Z)^{2n}$; assume that $\Gal(\tilde{L}/L_p(K))$ is generated by $\{\tau_1, \tau_2, \dots, \tau_{2n}\}$. Consider the $(\Z/p\Z)^2$-subgroups of $\Gal(\tilde{L}/L_p(K))$ generated by pairs of two elements in $\{\tau_1, \dots, \tau_{2n}\}$. There are $n(2n-1)$ of them; call them $A_1, A_2, \dots, A_{n(2n-1)}$. In particular, $A_{k+j} = \langle \tau_i, \tau_{i+j} \rangle$, where $i$ ranges from $1$ to $2n-1$ and $k = 2n(i-1) - \frac{(i-1)i}{2}$, $1 \leq j \leq 2n-i$. 

Note that $T = \frac{(p^{2n}-1)(p^{2n}-p)}{(p^2-1)(p^2-p)} \geq n(2n-1)$, for all primes $p$. From the fact that the decomposition subgroups $D_v$ exhaust the $(\Z/p\Z)^2$-subgroups of $\Gal(\tilde{L}/L_p(K))$, as $v$ ranges over all the primes in $\tilde{S}$, it follows that there are primes $v_i \in \tilde{S}$ such that $D_{v_i} \cong A_i$, for all $1 \leq i \leq n (2n-1)$. Note that these primes are primes above $g(\lambda_j), g(\alpha), g(\beta)$. Consider the following intersections:
\begin{align*}
	B_1 &= D_{v_1} \cap D_{v_2} \cap \dots \cap D_{v_{2n-1}} \cong \langle \tau_1 \rangle \cong \Z/p\Z \\
	B_2 &= D_{v_1} \cap D_{v_{2n}} \cap \dots \cap D_{v_{4n-3}} \cong \langle \tau_2 \rangle \cong \Z/p\Z  \\
	&\dots \\
	B_{2n} &=D_{v_{2n-1}} \cap D_{v_{4n-3}} \cap \dots \cap D_{v_{n(2n-1)}} \cong \langle \tau_{2n} \rangle \cong \Z/p\Z.
\end{align*}

The groups $B_i$ span the group $\Gal(\tilde{L}/L_p(K))$, so there exists a basis $\{ x_1, \dots x_{2n} \}$ of $\Gal(\tilde{L}/L_p(K))$ such that $x_i \in B_i$. Now,
\begin{align*}
	x_1, x_2 \in D_{v_1} &\Rightarrow x_1 \wedge x_2 \in \bigwedge^2 D_{v_1} \\
	x_1, x_3 \in D_{v_2} &\Rightarrow x_1 \wedge x_3 \in \bigwedge^2 D_{v_2} \\
	&\dots \\
	x_{2n-1}, x_{2n} \in D_{v_{n(2n-1)}} &\Rightarrow x_{2n-1} \wedge x_{2n} \in \bigwedge^2 D_{v_{n(2n-1)}},
\end{align*}
which implies that $\langle x_i \wedge x_j \mid i<j \rangle \subset \bigoplus \displaystyle \bigwedge^2 D_{v_i}$, where the sum is over all $v_i$ with $D_{v_i} \cong A_i$, for $1 \leq i \leq n(2n-1)$. On the other hand, $\displaystyle \langle x_i \wedge x_j \mid i<j \rangle$ spans $\displaystyle \bigwedge^2 \Gal(\tilde{L}/L_p(K))$, which implies that $\displaystyle \bigwedge^2 \Gal(\tilde{L}/L_p(K)) \subset \bigoplus \bigwedge^2 D_{v_i}  \subset \bigoplus \bigwedge^2 D_v$, so we must have equality. Thus, the map in Lemma \ref{lemma1} is surjective, proving that $L_p(\tilde{L}) = \tilde{L}$. 

Let $L_1 = K_1.L_p(K)$ and $L_2 = K_2.L_p(K)$. Then $\tilde{L}=L_1.L_2$. Note that $\Gal(L_1/F) \cong \Gal(L_2/F) \cong V \rtimes \Gamma \cong \Gamma^\prime$, by the discussion preceding Proposition \ref{embedding}. Let $K^\prime$ be the Galois closure of $K(\sqrt[p]{P_V (\nu_1 \alpha) P_V (\nu_2 \beta)^{-1}})$ over $F$. We would like to show that there exists $F^\prime/F$ that satisfies the conditions of Proposition \ref{prop2}. Since $\Gal(K^\prime/K) \cong V$ and $\Gal(K^\prime/F) \cong V \rtimes \Phi$ by construction, we have the following exact sequence
\[
1 \to \Gal(K^\prime/K) \to \Gal(K^\prime/F) \to \Gal(K/F) \to 1. 
\]
Recall that $V$ is a $p$-group and $\Phi$ is a group of order prime to $p$, so this sequence splits, and we can view $\Gal(K/F) = \Phi$ as a subgroup of $\Gal(K^\prime/F)$. Let $F^\prime = K^{\prime \Phi}$. The extension $K^\prime/F^\prime$ is Galois and has Galois group $\Phi$. Moreover, $K \subset K^\prime$ and $F \subset F^\prime$. By construction, $F^\prime/F$ and $L_1/F$ are linearly disjoint, and $L_1.F^\prime = L_1.K^\prime = \tilde{L}$, so $\Gal(\tilde{L}/F^\prime) \cong \Gal(L_1/F) \cong \Gamma^\prime$. 
We claim that $L_p(K^\prime) = \tilde{L}$. Since $\tilde{L}$ is an unramified $p$-extension of $L_p(K).K^\prime$ and $L_p(K).K^\prime$ is an unramified $p$-extension of $K^\prime$, it follows that $\tilde{L} \subset L_p(K^\prime)$. Combining this with the fact that $L_p(\tilde{L}) = \tilde{L}$, we conclude that $\tilde{L} = L_p(K^\prime)$. Moreover, every prime of $F^\prime$ lying over $p$ splits completely in $\tilde{L}$. To finish the proof in the case of a split extension, we have to prove that $L_p(K^\prime)/F^\prime$ satisfies property \textbf{P}. Since $L_p(K^\prime)/K^\prime$ is everywhere unramified, we only have to prove that the extension $K^\prime/F^\prime$ satisfies property \textbf{P}. Since $\Gal(K^\prime/F^\prime) \cong \Gal(K/F) \cong \Phi$ has order prime to $p$, and $F^\prime/F$ is a $p$-extension, this argument is similar to the one at the end of the proof of Lemma \ref{lemmma9}, and will be omitted.  \\

If the extension
\[
1 \to V \to \Gamma^\prime \to \Gamma \to 1
\]
is not split, we cannot work over $K$ anymore. However, the proof is similar. We begin by proving the following lemma, which is a variation of Lemma 6 in \cite{MR2772089}. In fact, we can recover Ozaki's lemma from the following result in the case of $\Phi = 1$, $V = \Z/p\Z$ and $\mu_p \subset K$. The proof is similar to the proof of Lemma 4 in the first version of Ozaki's paper, the main differences coming from the $\F_p$-dimension of $V$ and the action of $\Phi$ on $V$. 
\begin{lemma}\label{lemma4}
	For any group extension 
	\[
	1 \to V \to \Gamma^\prime \to \Gamma \to 1,
	\]
	there exists a finite extension $F^\prime/F$ such that if $K^\prime = K.F^\prime$, then 
	\begin{enumerate}
		\item $L_p(K^\prime) = L_p(K).K^\prime$ and $L_p(K^\prime) = L_p(K).F^\prime$; hence $\Gal(L_p(K^\prime)/F^\prime) \cong \Gal(L_p(K)/F)$ and every prime of $F^\prime$ lying over $p$ splits completely in $L_p(K^\prime)$.
		\item There exist global units $\varepsilon_1, \varepsilon_2, \dots, \varepsilon_n \in \mathcal{O}_{L_p(K^\prime)}^\times$ ($n = \dim_{\F_p} V$) such that the field extension $L_p(K^\prime)(\sqrt[p]{\varepsilon_1}, \dots, \sqrt[p]{\varepsilon_n})/F^\prime$ is Galois with Galois group isomorphic to $\Gamma^\prime$.  
	\end{enumerate}
\end{lemma}
\begin{proof}
	By Proposition \ref{embedding}, there exists a Galois extension $L/F$ containing $L_p(K)$ such that $\Gal(L/F) \cong \Gamma^\prime$. 
	Let $\alpha_1, \dots, \alpha_n$ be the Kummer generators of $L/L_p(K)$, so $L = L_p(K)(\sqrt[p]{\alpha_1}, \dots, \sqrt[p]{\alpha_n})$. Since $K \subset L_p(K) \subset L$ and $K/F$ is Galois, the extension $L/K$ must be Galois, so $(\alpha_i \mod{L_p(K)^{\times p}}) \in (L_p(K)^\times/L_p(K)^{\times p})^G$, where $G = \Gal(L_p(K)/K)$. Thus, $\alpha_i \mathcal{O}_{L_p(K)} = \mathfrak{A}_i^p \mathfrak{a}_i$, for $\mathfrak{A}_i$ an ideal of $\mathcal{O}_{L_p(K)}$ and $\mathfrak{a}_i$ an ideal of $\mathcal{O}_K$. Let $h$ be the class number of $L_p(K)$; so $(h,p) = 1$. Then, $\mathfrak{A}_i^h = A_i \mathcal{O}_{L_p(K)}$, for some $A_i \in L_p(K)$, which implies that $\alpha_i^h \mathcal{O}_{L_p(K)} = A_i^p \mathfrak{a}_i^h$. Let $\mathfrak{a}_i^{\prime} = \mathfrak{a}_i^h$ and $\alpha_i^{\prime} = \alpha_i^h A_i^{-p}$. Then $\mathfrak{a}_i^\prime = \alpha_i^\prime \mathcal{O}_{L_p(K)}$ is an ideal of $\mathcal{O}_K$, and $L = L_p(K)(\sqrt[p]{\alpha_1}, \dots, \sqrt[p]{\alpha_n}) = L_p(K)(\sqrt[p]{\alpha_1^\prime}, \dots, \sqrt[p]{\alpha_n^\prime})$. 

	Let $p^{e_i}$ be the exact power of $p$ dividing the order of the ideal class $[\alpha_i^\prime]_K$. Then $[\alpha_i^\prime]_K^{p^{e_i}}$ has order prime to $p$. Let $e =\max{e_i}$, and note that $[\alpha_i^\prime]_K^{p^e}$ has order prime to $p$. By using Proposition \ref{prop1} repeatedly we obtain an extension $F^\prime/F$ of degree $p^{e}$ such that if $K^\prime = K.F^\prime$, then 
	\begin{itemize}
		\item $K^\prime \cap L_p(K) = K$ and $F^\prime \cap L_p(K) = F$,
		\item $L_p(K^\prime) = L_p(K).K^\prime$ and $L_p(K^\prime) = L_p(K).F^\prime$,
		\item $\Gal(L_p(K^\prime)/F^\prime) \cong \Gamma$.
	\end{itemize}

	Denote by $j \colon Cl_K \to Cl_{K^\prime}$ the map induced by the inclusion $K \subset K^\prime$, and by $N \colon Cl_{K^\prime} \to Cl_K$ the norm map. The natural restriction $\Gal(L_p(K^\prime)/K^\prime) \to \Gal(L_p(K)/K)$, which is an isomorphism, induces an isomorphism $\Gal(L_p(K^\prime)/K^\prime)^{\text{ab}} \cong G^{\text{ab}}$. Then the order of $\ker N$ is prime to $p$ by class field theory. Thus, since $N \circ j ([\mathfrak{a}_i^\prime]_K) = [\mathfrak{a}_i^\prime]_K^{p^e}$, the order of $j([\mathfrak{a}_i^\prime]_K) \in Cl_{K^\prime}$, say $m_i$, is prime to $p$. Let $m = \text{lcm}(m_i)$, so $m$ is prime to $p$. Let $\mathfrak{a}_i^{\prime m} = a_i^\prime \mathcal{O}_{K^\prime}$, for some $a_i^\prime \in K^{\prime \times}$. Then $a_i^\prime \mathcal{O}_{K^\prime} = \alpha_i^{\prime m} \mathcal{O}_{L_p(K^\prime)}$. Thus, there exists $\varepsilon_i \in \mathcal{O}^\times_{L_p(K^\prime)}$ such that $\alpha_i^{\prime m} = a_i^\prime \varepsilon_i$. Note that $L_p(\sqrt[p]{\alpha_1^\prime}, \dots, \sqrt[p]{\alpha_n^\prime}) = L_p(\sqrt[p]{\alpha_1^{\prime m}}, \dots, \sqrt[p]{\alpha_n^{\prime m}})$, so we can replace $\alpha_i^\prime$ by $\alpha_i^{\prime m}$. 

	Recall that $\langle \alpha_i^\prime \mod{L_p(K^\prime)^{\times p}} \rangle \subset (L_p(K^\prime)^\times/L_p(K^\prime)^{\times p})^G$, as a subgroup. Since the extension $L_p(K^\prime)(\sqrt[p]{\alpha_1^\prime}, \dots, \sqrt[p]{\alpha_n^\prime})/F^\prime$ has Galois group $\Gamma^\prime$, it follows that $\langle \alpha_i^\prime \mod{L_p(K^\prime)^{\times p}} \rangle$ has an action of $\Gamma = G \rtimes \Phi$ and, as an $\F_p[\Gamma]$-representation, it is isomorphic to a copy of the projective cover of $V$; call it $\tilde{V}$. Hence, $\Gal(L_p(K^\prime)(\sqrt[p]{\alpha_1^\prime}, \dots, \sqrt[p]{\alpha_n^\prime})/F^\prime)$ is isomorphic to $\Gal(L_p(K^\prime)(\sqrt[p]{P_{\tilde{V}}(\alpha_1^\prime}), \dots, \sqrt[p]{P_{\tilde{V}}(\alpha_n^\prime}))/F^\prime)$, so we can replace $\alpha_i^\prime$ by $P_{\tilde{V}} (\alpha_i^\prime) = P_{\tilde{V}} (a_i \varepsilon_i) = P_{\tilde{V}}(a_i) \cdot P_{\tilde{V}}(\varepsilon_i)$. 
	Note that $P_{\tilde{V}}(a_i^\prime) = P_V (a_i^\prime)^{|G|}$, since $a_i^\prime \in K^{\prime \times}$. Since $\langle P_V(a_i^\prime) \rangle$ is isomorphic to a subrepresentation of $V$ in $K^{\prime \times}/K^{\prime \times p}$, and $V$ is an irreducible $\F_p[\Phi]$-representation, either $\langle P_V(a_i^\prime) \rangle \cong V$ or $\langle P_V(a_i^\prime) \rangle = 1$. The latter implies that $P_V(a_i^\prime) = 1$, for all $i$, which in turn implies that $P_{\tilde{V}} (\alpha_i^\prime) = P_{\tilde{V}} (\varepsilon_i)$, which finishes the proof: $L_p(K^\prime)(\sqrt[p]{P_V(\alpha_1^\prime)}, \dots, \sqrt[p]{P_V(\alpha_n^\prime)})/F^\prime$ is the desired extension.  On the other hand, the former implies that $\langle P_V(a_i^\prime) \rangle \cong V$, so $\Gal(K^\prime(\sqrt[p]{P_V(a_1^\prime)}, \dots \sqrt[p]{P_V(a_n^\prime)})/F^\prime) \cong V \rtimes \Phi$. Then, by Proposition \ref{embedding}, the Galois group of $L_p(K^\prime)(\sqrt[p]{P_V(a_1^\prime)}, \dots \sqrt[p]{P_V(a_n^\prime)})/F^\prime$ is isomorphic to $V \rtimes \Gamma$. Putting all this together, we can conclude that $\Gal(L_p(K^\prime)(\sqrt[p]{P_V(\varepsilon_1)}, \dots \sqrt[p]{P_V(\varepsilon_n)})/F^\prime) \cong \Gamma^\prime$. 
\end{proof}

Note that from construction, it follows that $\langle \varepsilon_1, \dots, \varepsilon_n \rangle \cong \langle \sigma (\varepsilon_1) \mid \sigma \in \Gamma \rangle$, as representations. Moreover, the new extension $L_p(K^\prime)(\sqrt[p]{\varepsilon_1}, \dots, \sqrt[p]{\varepsilon_n})/F^\prime$ satisfies property \textbf{P} except possibly at the primes above $p$. 

Just as before, construct primes $\lambda_i$, in $K^\prime$, for $1 \leq i \leq T$, and $\alpha, \beta$ in $K^\prime$ with the same properties. 
The primes $\alpha, \beta, \lambda_i$ are in $K^\prime$, so just as in the previous case, the Galois closure of $K_1 = K^\prime(\sqrt[p]{P_V(\nu_1\alpha)})$ over $F$ has Galois group $V \rtimes \Phi$. Define the field $L_1 = L_p(K^\prime)(\sqrt[p]{\varepsilon_1 P_V(\nu_1\alpha)}, \dots, \sqrt[p]{\varepsilon_n g_n(P_V(\nu_1\alpha))})$, and use Proposition \ref{embedding} and Lemma \ref{lemma4} to observe that the Galois group of $L_1/F^\prime$ is $\Gamma^\prime$. Construct $K_2$ and $L_2$ in a similar manner. Let $\tilde{K}$ be the Galois closure of $K^\prime(\sqrt[p]{P_V(\nu_1\alpha) P_V(\nu_2\beta)^{-1}})$ over $F^\prime$ and let $\tilde{L} = L_1.L_2$. Construct a field $\tilde{F}$ just as in the previous case. Then the new extension $\tilde{L}/\tilde{F}$ satisfies the conditions of Proposition \ref{prop2}. Note that property \textbf{P} is now satisfied, since there is no ramification at the prime $p$.

\section{Proof of Theorem \ref{thm2}} \label{proofofthm2}

Let $R$ be any local ring admitting a surjection to $\Z_p$ with finite (as a set) kernel $I_R$. Assume we have an absolutely irreducible residual representation 
\[
\overline{\psi} \colon G_F \to \GL_2(\F_p),
\]
whose image is $\tilde{\Phi}$, a group of order prime to $p$. Note that the projective image $\Phi$ of this representation is $A_4$, $S_4$, $S_5$ or a dihedral group. Since $\tilde{\Phi}$ has order prime to $p$, it lifts to a representation $\tilde{\Phi} \subset \GL_2(\Z_p)$. Let $\tilde{\Gamma}$ denote the inverse image of this group inside $\GL_2(R)$; it lives inside a split exact sequence
\[
1 \to 1 + M_2(I_R) \to \tilde{\Gamma} \to \tilde{\Phi} \to 1.
\]
The group $\tilde{\Gamma}$ admits a natural residual representation via $\overline{\psi}$, call it $\overline{\rho} \colon \tilde{\Gamma} \to \GL_2(\F_p)$, which is absolutely irreducible, so a universal deformation ring $R_{\overline{\rho}}$ exists. The aim of this section is to show that $R_{\overline{\rho}} \cong R$. 

Firstly, note that $\tilde{\Gamma}$ admits a deformation to $\GL_2(R)$ by construction, so $R_{\overline{\rho}} \twoheadrightarrow R$, so there exists an ideal $J \subset R_{\overline{\rho}}$ such that $R_{\overline{\rho}}/J \cong R$. The general case can be reduced to the case when $J$ is finite, so from now on, we will assume that $J$ is finite.
To prove that $R$ is universal, it is enough to prove that for every local ring $S$ with the following properties 
\begin{itemize}
	\item There exists a finite ideal $I_S$ such that $S/I_S \cong \Z_p$;
	\item $S \twoheadrightarrow R$;
	\item There is a lift $\tilde{\Gamma} \to \GL_2(S)$,
\end{itemize}
there exists a map $R \to S$ that makes the following diagram commute: 
\begin{center}
	\begin{tikzcd}
  \tilde{\Gamma} \arrow[r] \arrow[dr]
    & \GL_2(S) \arrow[d, two heads, shift left=1.5ex]\\
&\GL_2(R) \arrow[u]\end{tikzcd}
\end{center}

Before proving our main result, we need to introduce several lemmas. From now on, assume that $R$ and $S$ are local rings with the following properties:
\begin{enumerate}
	\item There exist finite ideals $I_S \subset S$ and $I_R \subset R$ such that $R/I_R \cong S/I_S \cong \Z_p$.
	\item $S \twoheadrightarrow R$.
\end{enumerate}

Since $S \twoheadrightarrow R$ and the ideals $I_S$ and $I_R$ are finite, it follows that there exists a finite ideal $J \subset S$ such that $S/J \cong R$.

We will use the following notation:
\begin{itemize}
	\item If $A$ is a local ring, let $\m_A$ denote the maximal ideal of $A$.
	\item Let $\Gamma_S = 1+ \M_2(I_S)$.
	\item For an ideal $I \subset S$, let $\Gamma_I = 1+ \M_2(I)$. 
\end{itemize}

\begin{lemma}\label{subclaim3}
	Let $S$ and $R$ be as above. Suppose that $\dim_{\F_p} J = 1$ and suppose that there exists some $0 \neq x \in J$ such that $x$ is sent to $0$ under the map $\m_S/(\m_S^2, p) \to \m_R/(\m_R^2, p)$. Then there exists a subring $S^\prime \subset S$ such that $S^\prime \cong R$. 
\end{lemma}
\begin{proof}
	Take generators $\overline{x_1}, \dots, \overline{x_d}$ of $\m_R/(\m_R^2, p)$ and lift these generators to $\m_S/(\m_S^2, p)$ and then to $S$. Denote the lifts to $S$ by $x_1, \dots, x_d$. Consider the subring $S^\prime$ of $S$ generated by these lifts. 

	Firstly, we claim that $x \not\in S^\prime$. Suppose otherwise, so $x = a_0 + a_1x_1 + \dots a_d x_d + \alpha$, where $\alpha \in \m_S^2$, $a_i \in \Z_p$. Since $x, x_i, \alpha \in \m_S$, it follows that $a_0 \in \m_S \cap \Z_p = (p)$. Recall that under the map
	\[
	\m_S/(\m_S^2, p) \to \m_R/(\m_R^2, p), 
	\]
	$x = a_0 + \sum a_i x_i + \alpha = \sum a_i x_i \text{ (in } \m_S/(\m_S^2, p))$ is sent to $0$. 
	Since $\overline{x_i}$ generate $\m_R/(\m_R^2, p)$, this is true if and only if $a_i = 0 \in \m_R/(\m_R^2, p)$, for all $i$, which is true if and only if $a_i \in \Z_p \cap (\m_R^2, p)= (p)$. Thus, $x \in (\m_S^2, p)$, which is a contradiction. So $x \not\in S^\prime$, which implies that $S^\prime \hookrightarrow R$.

	On the other hand, there is an isomorphism on cotangent spaces induced by this inclusion $S^\prime \hookrightarrow R$. In particular, the map on cotangent spaces is surjective, which implies that the original map must be surjective. Thus, the map from $S^\prime$ to $R$ must be an isomorphism. 
\end{proof}

\begin{lemma}\label{lemma}
	Let $S$ and $R$ be as above and suppose that $\dim_{\F_p} J = 1$. Then exactly one of the following holds:
	\begin{itemize}
		\item $[\Gamma_S, \Gamma_S]\Gamma_S^p$ contains $\Gamma_J$.
		\item $S \cong R[x]/(x^2, px)$, for some $x \in S$. 
	\end{itemize}
\end{lemma}
\begin{proof}
	Note that $[\Gamma_S, \Gamma_S]\Gamma_S^p = \Gamma_{(I_S^2, pI_S)}$, so if $J \subset (I_S^2, pI_S)$, then the first statement holds. Assume now that $J$ is not a subset of $(I_S^2, pI_S)$. So, there exists $x \in J$ such that $x \not\in (I_S^2, pI_S)$. Moreover, $x \not\in (\m_S^2, p)$. 
	Thus, under the map
	\[
	\m_S/(\m_S^2, p) \to \m_R/(\m_R^2, p),
	\]
	the nonzero element $x$ is sent to $0$. From Lemma \ref{subclaim3}, we know that there exists a subring $S^\prime \subset S$ such that $S^\prime \cong R$ and $x \not\in S^\prime$. Consider the map $S^\prime[x] \to S$: it is a surjection and $px, x^2$ are elements of the kernel. Thus $S \cong S^\prime[x]/(x^2, px) \cong R[x]/(x^2, px)$. 
\end{proof}

\begin{lemma} \label{subclaim4}
	Let $S$ and $R$ be as above. If $\Gamma_J \subset [\Gamma_S, \Gamma_S]\Gamma_S^p$, then there is no lift $\tilde{\Gamma} \to \GL_2(S)$. 
\end{lemma}
\begin{proof}
	Suppose that $\Gamma_J \subset [\Gamma_S, \Gamma_S]\Gamma_S^p$ and suppose that there is a lift $\tilde{\Gamma} \to \GL_2(S)$. Then, we have the following commutative diagram:
	\begin{center}
		\begin{tikzcd}[column sep=tiny]
& \Gamma_S \ar[dr, "\psi", two heads] 

& \\
  \Gamma_R \ar[ur] \ar[dr, "\simeq"]
    &
      & F(\Gamma_S)\\
& \Gamma_R  \ar[ur, dashed, two heads]
&
&
\end{tikzcd}
	\end{center}
	where $F(\Gamma_S) = \Gamma_S/[\Gamma_S, \Gamma_S]\Gamma_S^p$ is the Frattini quotient. 

	Since $\Gamma_J \subset [\Gamma_S, \Gamma_S]\Gamma_S^p = \ker \psi$, it follows that $\psi$ has to factor through $\Gamma_R$, which implies that $\Gamma_R \cong \Gamma_R \to F(\Gamma_S)$ must be surjective, so $\Gamma_R \to \Gamma_S \to F(\Gamma_S)$ must be surjective. Since $F(\Gamma_S)$ is the Frattini quotiont of $\Gamma_S$, we obtain that $\Gamma_R \to \Gamma_S$ must be surjective, but this is impossible, since $\Gamma_R$ has fewer elements than $\Gamma_S$. Therefore, there is no lift to $\GL_2(S)$ in this case. 
\end{proof}

\begin{lemma}\label{subclaim6}
	Let $R$ and $S$ be as above. Assume, moreover, that $J$ is an $\F_p$-vector space of dimension $n$. Then exactly one of the following is true:
	\begin{enumerate}
		\item There is no lift $\tilde{\Gamma} \to \GL_2(S)$;
		\item $S \cong R[x_1, \dots, x_n]/(x_ix_j, px_i)$, for some $x_i \in S$.
	\end{enumerate}
\end{lemma}
\begin{proof}
	This will be proved by induction on $n = \dim_{\F_p}J$. 

	Base case: $n = 1$. By Lemma \ref{lemma}, we know that exactly one of the following holds: 
	\begin{enumerate}
		\item $\Gamma_J \subset [\Gamma_S, \Gamma_S]\Gamma_S^p$.
		\item $S \cong R[x]/(x^2, px)$, for some $x \in S$. 
	\end{enumerate}
	If $(1)$ is true, then by Lemma \ref{subclaim4}, there is no lift $\tilde{\Gamma} \to \GL_2(S)$. If $(2)$ is true, there is nothing to prove.  This concludes the base case. 

	Inductive step: suppose the result is true for $m < n$ and consider the case $n = \dim_{\F_p}J$. Suppose $J = (x_1, \dots, x_n)$, for some $x_i \in S$. Let $S_1 = S/(x_1)$. Apply the base case to $S$, $S_1$ and $(x_1)$ instead of $S$, $R$ and $J$ to obtain that either there is no lift $\tilde{\Gamma} \to \GL_2(S)$ or $S \cong S_1[x_1]/(x_1^2, px_1)$. If the former is true, then the proof is complete. Assume that $S \cong S_1[x_1]/(x_1^2, px_1)$. By construction, $S_1$ still surjects onto $R$ with kernel $J_1 = J/(x_1)$. Note that $\dim_{\F_p} J_1 < n = \dim_{\F_p} J$. By induction, we know that $S_1 \cong R[x_2, \dots x_n]/(x_i x_j, px_i)$. Putting these two things together, we obtain that $S \cong R[x_1, \dots x_n]/(x_ix_j, px_i)$. 
\end{proof}

\begin{proposition}
	Let $R$ and $S$ be two local rings as above. Then there exists a splitting $R \to S$. 
\end{proposition}
\begin{proof}
	Let $I_R$ be the kernel of $R \twoheadrightarrow \Z_p$, $I_S$ be the kernel of $S \twoheadrightarrow \Z_p$, and $J$ be the kernel of $S \twoheadrightarrow R$. 
	We will prove this result by induction on $\ell(S)$, where we define $\ell(S) = \ell(I_S)$, the length of the ideal $I_S$. 

	Base case: $\ell(S) = \ell(R) + 1$. Replacing $S$ by $S/\m_S J$, if necessary, we can assume that $J$ is an $\F_p$-vector space. We can do this, because to find a lift to $S$, it is enough to find a lift to $S/\m_S J$. Then the condition $\ell(S) = \ell(R) + 1$ translates to $\dim_{\F_p} J = 1$, and this follows from the base case in Lemma \ref{subclaim6}.

	Inductive step: suppose true for $\ell(S) < N$; we would like to prove that for $\ell(S) = N$ there is a splitting $R \to S$. If $S \twoheadrightarrow R$ has kernel $J$, then
	\begin{align*}
		S \xrightarrow{\pi} S/\m_SJ \twoheadrightarrow S/J = R. 
	\end{align*}
	Let $S^\prime = S/\m_SJ$. Then $S^\prime$ is a local $\Z_p$-algebra with maximal ideal $\m_{S^\prime} = \m_S/\m_SJ$.  Moreover, there is a surjection from $S^\prime$ to $R$ with kernel $J^\prime = J/\m_SJ$. This new local $\Z_p$-algebra has the following properties:
	\begin{itemize}
		\item $S^\prime/J^\prime = R$;
		\item $\m_{S^\prime}J = 0$;
		\item $J^\prime$ is a $S/\m_S = \F_p$-vector space.
	\end{itemize}
	Thus, by Lemma \ref{subclaim6}, it follows that exactly one of the following is true
	\begin{enumerate}
		\item There is no lift $\tilde{\Gamma} \to \GL_2(S^\prime)$;
		\item $S^\prime \cong R[x_1, \dots, x_n]/(x_ix_j, px_i)$, for some $x_i \in S^\prime$.
	\end{enumerate}
	If there is no lift to $\GL_2(S^\prime)$, then there's no lift to $\GL_2(S)$, so there's nothing to be proved. The second case gives us a splitting $R \hookrightarrow S^\prime$. We are in the following situation: 
	\begin{center}
		\begin{tikzcd}
  S \arrow[r, two heads, "\pi"]
    & S^\prime \arrow[r, two heads] 
    & R \\
&R \arrow[u, hook]\end{tikzcd}
	\end{center}
	Let $S^{\prime\prime} = \pi^{-1}(R) \subset S$. Then $\ell(S^{\prime\prime}) < \ell(S)$, so by induction there is a lift $R \to S^{\prime\prime}$. Then we can conclude that there is a lift
	\begin{align*}
		R \to S^{\prime\prime} \to S,
	\end{align*}
	where the first map comes from the induction step and the second map comes from the definition of $S^{\prime \prime}$. 
\end{proof}

We can now prove Theorem \ref{thm2}. Recall that at the beginning of this section we assumed the existence of an absolutely irreducible residual representation $\overline{\psi} \colon G_K \to GL_2(\overline{\F_p})$, whose image is $\tilde{\Phi}$. From this, we obtained a short exact sequence 
\[1 \to G \to \tilde{\Gamma} \to \tilde{\Phi} \to 1,\] and a residual representation $\overline{\rho} \colon \tilde{\Gamma} \to \GL_2(\overline{\F_p})$. We have already proved that $R$ is the universal deformation ring of $\overline{\rho}$. To prove that $R$ is a universal everywhere unramified deformation ring of some residual representation, we need to find extensions $H/K/F$ with $\Gal(H/F) = \tilde{\Gamma}$ and $H/K$ the maximal everywhere unramified pro-$p$ extension of $K$. The existence of such extensions is given by Theorem \ref{thm1}, under the assumption that there exists a $\tilde{\Phi}$-extension of $\Q(\zeta_p)$ with class number prime to $p$ that satisfies property $\textbf{P}$. We claim that we can reduce the problem to $\Phi$, the projective image of $\tilde{\Phi}$: consider the following exact sequence \[1 \to G \to \tilde{\Gamma} \to \tilde{\Phi} \to 1\] and its projective image \[1 \to G \to \Gamma \to \Phi \to 1.\] We are in the following case
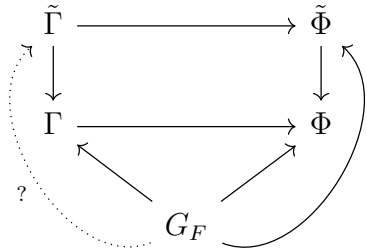
\[
    \begin{tikzcd}
    	\tilde{\Gamma} \arrow[rr]  \arrow[d] & &\tilde{\Phi} \arrow[d] \\
    	\Gamma \arrow[rr] & & \Phi \\
    	& G_F \arrow[uur, bend right=80] \arrow[ur] \arrow[ul] \arrow[uul, dotted, bend left = 80, "?"]
	\end{tikzcd}
\]
To prove that there exists a lift $G_F \to \tilde{\Gamma}$, take two compatible set theoretic lifts. The centres $Z(\tilde{\Gamma}) = Z(\tilde{\Phi})$ are equal, so the 2-cocyles will be the same. Since any set theoretic lift to $\tilde{\Phi}$ is a homomorphism, it follows that the lift to $\tilde{\Gamma}$ must be a homomorphism, and we are done.

To finish the proof of Theorem \ref{thm2}, we constructed extensions with the desired properties for $p=5$ and $p=7$ (see the two examples below) using GP/Pari (\cite{GPpari}) and the Database of Number Fields \url{https://hobbes.la.asu.edu/NFDB/} (\cite{MR3356048}). Note that this proof works for any regular prime $p \geq 5$ under the extra assumption that there exists a $\Phi$-extension of $\Q(\zeta_p)$ with class number prime to $p$ that satisfies property \textbf{P}. 

\begin{ex}
	Let $p=5$ and $\tilde{\Phi} = \tilde{A_4} = \SL_2(\F_3)$ (so $\Phi = A_4$). Let $E = \Q(\zeta_5)$ and let $\tilde{L_1}$ be the Galois closure of the field defined by the following polynomial over $\Q$:
	\begin{align*}
		x^8 - 5x^6 - 3x^5 + 28x^4 - 12x^3 - 80x^2 + 256.
	\end{align*}
	This field has an intermediate field $L_1$, which is the Galois closure of the field defined by
	\begin{align*}
		x^4 - 21x^2 - 3x + 100. 
	\end{align*}
	Let $\tilde{L} = E.\tilde{L_1}$ and $L = E.L_1$. Then $L$ is a subfield of $\tilde{L}$ and $\Gal(\tilde{L}/E) = \SL_2(\F_3) \twoheadrightarrow A_4 = \Gal(L/E)$. We used GP/Pari \cite{GPpari} to show that the class numbers of both $L$ and $\tilde{L}$ are prime to $5$ and that both $L/E$ and $\tilde{L}/E$ satisfy property \textbf{P}. This concludes the proof for $p = 5$. 
\end{ex}

\begin{ex}
	Let $p=7$ and $\tilde{\Phi} = \tilde{A_4} = \SL_2(\F_3)$ (so $\Phi = A_4$). Let $E = \Q(\zeta_7)$ and let $\tilde{L_1}$ be the Galois closure of the field defined by the following polynomial over $\Q$:
	\begin{align*}
		x^8 - x^7 - 11x^6 + 13x^5 + 32x^4 - 41x^3 - 23x^2 + 32x - 1.
	\end{align*}
	This field has an intermediate field $L_1$, which is the Galois closure of the field defined by
	\begin{align*}
		x^4 - x^3 - 11x^2 + 4x + 12.
	\end{align*}
	Let $\tilde{L} = E.\tilde{L_1}$ and $L = E.L_1$. Then $L$ is a subfield of $\tilde{L}$ and $\Gal(\tilde{L}/E) = \SL_2(\F_3) \twoheadrightarrow A_4 = \Gal(L/E)$. We used GP/Pari \cite{GPpari} to show that the class numbers of both $L$ and $\tilde{L}$ are prime to $7$ and that both $L/E$ and $\tilde{L}/E$ satisfy property \textbf{P}. This concludes the proof for $p = 7$. 
\end{ex}

\section*{Acknowledgements}
I would like to thank my PhD advisor Frank Calegari for suggesting this problem, and for the constant support and helpful discussions on this topic. I also want to thank Professor Ravi Ramakrishna for taking the time to read the first draft of this paper and for all his valuable suggestions. I would also like to thank Sam Quinn for the useful comments and discussions throughout the project. Finally, I would like to thank Gal Porat for the help provided when I first started working on the proof of Proposition \ref{prop2}.

\bibliographystyle{siam}
\bibliography{paper}

\end{document}